\documentclass{article}
\usepackage[utf8]{inputenc}
\usepackage{geometry}[margin=1in]
\usepackage{amsmath}
\usepackage{nicematrix}
\usepackage{amssymb}
\usepackage{amsthm}
\usepackage{enumitem}
\usepackage{pgfplots}
\usepackage{svg}
\usepackage{graphicx}
\usepackage{comment}
\usepackage{tikz}
\usepackage{tikz-cd}
\usepackage{indentfirst}
\usepackage{titlesec}
\usepackage[noadjust]{cite}
\usepackage{stmaryrd}
\usepackage[pagebackref,hidelinks]{hyperref}
\usepackage{pinlabel}
\hypersetup{colorlinks=true, linkcolor = blue, citecolor = blue}
\usepackage{graphicx,import}
\usepackage{placeins}

\newcommand{\zz}{\ensuremath{\mathbb{Z}}}

\newcommand{\ff}{\ensuremath{\mathbb{F}}}

\newcommand{\bb}[1]{\mathbb{#1}}
\newcommand{\del}{\partial}

\renewcommand{\cal}[1]{\mathcal{#1}}

\newcommand{\spinc}{\text{spin}^\text{c}}
\newcommand{\pmd}{\ensuremath{{\pm\circ}}}

\DeclareMathOperator{\Mod}{Mod}

\theoremstyle{definition}
\newtheorem{thm}{Theorem}[section]
\newtheorem{defn}[thm]{Definition}
\newtheorem{lem}[thm]{Lemma}
\newtheorem{cor}[thm]{Corollary}
\newtheorem{prop}[thm]{Proposition}
\newtheorem{rem}[thm]{Remark}

\newtheorem*{ack}{Acknowledgements}
\newtheorem{mainthm}{Theorem}

\newtheorem{subthm}{Theorem}

\titleformat{\section}{\Large\bfseries\filcenter}{\thesection.}{1em}{}

\title{Fixed-point-free pseudo-Anosovs, and genus-two L-space knots in the Poincar\'e sphere}
\author{Braeden Reinoso}
\date{}

\begin{document}

\maketitle

\begin{abstract}
We classify genus-two L-space knots in the Poincar\'e homology sphere. This leads to the second knot Floer homology detection result for a knot of genus at least two, and the first such result outside of $S^3$. The argument uses the theory of train tracks and folding automata, and builds off of recent work connecting knot Floer homology of fibered knots to fixed points of mapping classes. As a consequence of our proof technique, we almost completely classify genus-two, hyperbolic, fibered knots with knot Floer homology of rank 1 in their next-to-top grading in any 3-manifold.
\end{abstract}

\section{Introduction}

The Poincar\'e homology sphere $\cal{P}=\Sigma(2,3,5)$ is a small Seifert-fibered space with three exceptional fibers of orders 2, 3, and 5, respectively. Other than $S^3$, the Poincar\'e sphere is the only known irreducible $\bb{Z}$-homology sphere which is also a Heegaard Floer L-space. Because the Floer theory of $\cal{P}$ is so simple, many Floer-theoretic questions about $S^3$ have natural analogues in $\cal{P}$. This paper will explore the analogue of a surgery question that was recently answered for $S^3$ in \cite{FRW}: which genus-two knots admit non-trivial surgeries to Heegaard Floer L-spaces?

Consider the 5-braid $\beta=(\sigma_1\sigma_2\sigma_3\sigma_4)^3$, which has braid closure the (3,5)-torus knot $\widehat{\beta}=T(3,5)$. Thinking of the Poincar\'e homology sphere $\cal{P}$ as the double cover over $S^3$ branched along $T(3,5)$, we may lift the braid axis for $\widehat{\beta}=T(3,5)$ to a knot $\cal{K}\subset\cal{P}$. The knot $\cal{K}$ is naturally a fibered knot, and appears as the order-three fiber of the Seifert fibration of $\cal{P}$. Moreover, $\cal{K}$ is an L-space knot: its exterior is Seifert-fibered with two exceptional fibers, so Dehn filling along the fibration slope produces a Lens space. We will prove:

\begin{mainthm}\label{thm:lspace}
$\cal{K}$ is the only genus-two L-space knot in $\cal{P}$ with irreducible exterior.
\end{mainthm}

The assumption that the knot exterior be irreducible is not crucial. Because $\cal{P}$ itself is irreducible, any knot in $\cal{P}$ with reducible exterior is contained in a 3-ball. So, to understand L-space knots in $\cal{P}$ with reducible exterior, it suffices to study L-space knots in $S^3$.

\begin{cor}
$\cal{K}$ and $T(2,5)\subset B^3\subset\cal{P}$ are the only genus-two L-space knots in $\cal{P}$.
\end{cor}
\begin{proof}
Let $K\subset\cal{P}$ be a genus-two L-space knot with reducible exterior. Because $\cal{P}$ is irreducible, we can see that $K\subset B^3\subset\cal{P}$. If $K$ has surgery to an L-space $Y$, we may write $Y=\cal{P}\#Y'$, where $Y'$ is surgery on $K$ thought of as a knot in $B^3\subset S^3$. Since $Y$ is an L-space, $Y'$ must be an L-space as well. In this case, $K$ is isotopic in $B^3\subset S^3$ to a genus-two L-space knot, and the only genus-two L-space knot in $S^3$ is $T(2,5)$, by \cite{FRW}.
\end{proof}

\begin{cor}
Knot Floer homology detects $\cal{K}$ and $T(2,5)\subset B^3\subset\cal{P}$.
\end{cor}

The exterior of any knot $K\subset\cal{P}$ is reducible, hyperbolic, Seifert-fibered, or toroidal. The only knots in $\cal{P}$ with Seifert fibered exterior are the three exceptional fibers and the (unique) regular fiber of the standard Seifert fibration of $\cal{P}$. Jacob Caudell \cite{J} has computed the genera of all four of these knots, and $\cal{K}$ is the only one with genus two. And, an Alexander polynomial argument (see subsection \ref{sec:floer}) shows that any genus-two L-space knot with irreducible exterior in $\cal{P}$ has atoroidal exterior. So, to prove Theorem \ref{thm:lspace}, it suffices to classify genus-two hyperbolic L-space knots in $\cal{P}$.

To do that, we will use an argument inspired by \cite{FRW}, which builds on work of \cite{BHS}. The key idea is that the study of hyperbolic L-space knots in genus two is closely related to the study of fixed points of pseudo-Anosov maps on the genus-two surface.

\begin{defn}
Let $K$ be a hyperbolic, fibered knot in a 3-manifold $Y$. The fibration of the exterior of $K$ is described by an open book decomposition $(S,h)$, where $S$ is a compact surface with one boundary component, and $h:S\to S$ is freely isotopic to a pseudo-Anosov map $\psi_h$ on $S$. We say that $K$ is \textit{fixed-point-free} (FPF for short) if $\psi_h$ has no fixed points in the interior of $S$.
\end{defn}

The following theorem will be the main goal of this paper:

\begin{mainthm}\label{thm:FPF}
Let $K$ be a genus-two, hyperbolic, fibered knot in $\cal{P}$. If the fractional Dehn twist coefficient $c(K)\neq0$, then $K$ is not FPF.
\end{mainthm}

Theorem \ref{thm:FPF} implies Theorem \ref{thm:lspace} as described in subsection \ref{sec:floer}. The argument uses well-known results from knot Floer homology (as translated by Tange in \cite{Tange} to the setting of $\cal{P}$) and more recent developments by \cite{BHS}, \cite{GS}, and \cite{Ni2} relating the knot Floer homology of a fibered knot to fixed points of its monodromy. With this machinery in place, the proof of Theorem \ref{thm:lspace} from Theorem \ref{thm:FPF} will be fairly quick.

So, most of the paper will be devoted to proving Theorem \ref{thm:FPF}. The proof is split into four subtheorems: Theorems \ref{thm:433}, \ref{thm:6}, \ref{thm:244}, and \ref{thm:2-34}. The first two are very short, following work from \cite{FRW}; the hardest case is Theorem \ref{thm:2-34}. An outline of the complete argument for Theorem \ref{thm:FPF} is provided below, in subsection \ref{sec:outline}. However, it may be helpful to first read the background section (Section \ref{sec:background}) for a review of the geometry of surface maps, or for our conventions.

Before continuing to the outline, we first update the classification of FPF knots in various genus-two strata as initiated in \cite{FRW}. We expect this classification to be helpful for knot Floer homology detection results and botany questions in 3-manifolds other than $S^3$ or $\cal{P}$.

\begin{thm}\label{thm:classify}
Let $K$ be a genus-two, fibered, hyperbolic knot with monodromy $h:S\to S$ and geometric representative $\psi_h$. If $\psi_h$ has no interior fixed points then one of the following is true:
\begin{itemize}
    \item $\psi_h$ has 1-pronged or 6-pronged boundary
    \item $h$ or $h^{-1}$ is conjugate to the lift of $\Delta^{4k+2}\sigma_1^{n+2}\sigma_2\sigma_3\sigma_4\sigma_1\sigma_2\sigma_3\sigma_4^2$ for some $n\geq 0,k\in\bb{Z}$
    \item $h$ or $h^{-1}$ is conjugate to the lift of $\Delta^{4k+2}(\sigma_4\sigma_3)^2(\sigma_2\sigma_1)^{-2}$ for some $k\in\bb{Z}$
    \item $h$ or $h^{-1}$ is conjugate to the lift of $\Delta^{4k+2}\sigma_1^{-3}\sigma_2^{-1}\sigma_3^{-1}\sigma_2(\sigma_3\sigma_4)^2$ for some $k\in\bb{Z}$
    \item $h$ or $h^{-1}$ is conjugate to the lift of $\Delta^{4k+2}(\sigma_4\sigma_3\sigma_1^{-1}\sigma_2^{-1})^2$ for some $k\in\bb{Z}$
\end{itemize}
\end{thm}

After applying \cite{Nifibered}, \cite{Ni2} or \cite{GS}, \cite{Ni3}, and some of the techniques from \cite{FRW}, Theorem \ref{thm:classify} translates to Floer homology as follows:

\begin{cor}\label{cor:botany}
Let $K$ be a null-homologous knot of genus-two in a 3-manifold $Y$. Suppose that the exterior of $K$ is hyperbolic, and denote by $S$ a genus-two Seifert surface for $K$. If $\text{rk}~\widehat{\mathit{HFK}}(K,[S], a)=1$ for $a=1,2$ then $K$ is fibered, and either:
\begin{itemize}
    \item The invariant foliations of its monodromy have 6-pronged boundary, and the dilatation of its monodromy is a root of $\Delta_K(t)=t^4-t^3\pm|H_1(Y)|t^2-t+1$
    \item The monodromy of $K$ is conjugate to one of the maps in Theorem \ref{thm:classify}
\end{itemize}
\end{cor}

The techniques used to prove Theorem \ref{thm:classify} can surely be extended to the case where $\psi_h$ has 6-pronged boundary, and should not pose much greater difficulty than the already completed cases. Completing the classification in the spirit of Theorem \ref{thm:classify} for maps with 6-pronged boundary would yield a complete solution to the botany problem described in Corollary \ref{cor:botany}: a list of all genus-two hyperbolic knots in any 3-manifold, with knot Floer homology of rank 1 in the top two gradings. Incidentally, that work would also lead to a complete classification of pseudo-Anosov maps with a unique fixed point on the closed genus-two surface. But, that is beyond the goal of the current paper; we may return to it in future work. For those interested, we encourage you to see Section 3 of \cite{FRW} for a discussion of the 6-pronged case.

\subsection{Outline of the argument for Theorem \ref{thm:FPF}}\label{sec:outline}
Let $K$ be a genus-two FPF knot in $\cal{P}$ with fractional Dehn twist coefficient $c(K)\neq 0$; our goal will be to find a contradiction to show that such a $K$ cannot exist. Associated to $K$, there is an open book decomposition $(S,h)$ for $\cal{P}$. Here, $S$ is a genus-two surface with one boundary component, and $h:S\to S$ is a homeomorphism fixing $\partial S$ pointwise. Because $K$ is hyperbolic, the monodromy $h:S\to S$ is freely isotopic to a pseudo-Anosov map $\psi_h: S\to S$.

Because $\cal{P}$ is an L-space, we know that $|c(K)|<1$ by \cite{HM} or \cite{OScont}. So, we must have $0<|c(K)|<1$, and hence $c(K)$ is not an integer. It follows that the foliations preserved by $\psi_h$ have at least two prongs on $\partial S$. By work of Baldwin--Hu--Sivek in \cite{BHS}, we may then conclude that $h$ commutes with a hyperelliptic involution $\iota:S\to S$. Taking a quotient by the action of $\iota$, we see that $h$ projects to a 5-braid $\beta$, thought of as a mapping class on the 5-marked disk $D_5$. Equivalently, $h$ is the lift of $\beta$ via the Birman--Hilden correspondence. From this perspective, we can see that $\cal{P}$ is a 2-fold cover over $S^3$ branched along the braid closure $\widehat{\beta}$. In particular, $\widehat{\beta}$ must be the torus knot $T(3,5)$. And, because $|c(K)|<1$, we know $|c(\beta)|<2$.

We can see that $\beta$ is freely isotopic to a pseudo-Anosov map $\psi_\beta:D_5\to D_5$, which lifts to $\psi_h:S\to S$. The invariant foliations of $\psi_\beta$ lift to those of $\psi_h$, and, near $\partial D_5$, the number of prongs of these foliations double in the lift. It follows that $\psi_h$ has an even number of prongs on $\partial S$. Applying the Euler-Poincar\'e formula (see e.g. \cite{FLP}), the possible strata for $\psi_h$ are then: $(6;\emptyset; \emptyset)$, $(4;\emptyset;4)$, $(4;\emptyset;3^2)$, $(2;\emptyset;4^2)$, or $(2;\emptyset;3^4)$. See subsection \ref{sec:pa} for strata naming conventions.

Note that any map in the stratum $(4;\emptyset;4)$ has an interior fixed point given by the unique interior 4-pronged singularity, so $\psi_h$ cannot belong to this stratum. In joint work with Farber and Wang in \cite{FRW}, we classified pseudo-Anosov maps with no interior fixed points in the stratum $(4;\emptyset;3^2)$, and we can check quickly that no such map can specify an open book decomposition for $\cal{P}$; see Theorem \ref{thm:433}. Similarly, we gave strong restrictions for fixed-point-free maps in the stratum $(6;\emptyset;\emptyset)$, which can easily be applied to our setting; see Theorem \ref{thm:6}. We will treat these strata in Section \ref{sec:6433}, and will ignore them for the rest of this outline.

For the remaining two strata, $(2;\emptyset;4^2)$ and $(2;\emptyset;3^4)$, the relevant results are Theorems \ref{thm:244} and \ref{thm:2-34}. The arguments for each are similar, so we will discuss them in parallel here. For these two strata, we will work mostly with the braid $\beta$. We can use the singularity type for $\psi_h$ to determine the singularity type for $\psi_\beta$, by considering how the foliations project in the quotient under the action of $\iota$; the relevant strata for $\psi_\beta$ are $(1;1^5;4)$ and $(1;1^5;3^2)$. In each of these strata, we consider the possible train tracks carrying $\psi_\beta$. Applying Theorem \ref{thm:jointless} and some techniques from the folding automata developed in \cite{KLS}, \cite{HS} and \cite{CH}, we show that it suffices to consider a single train track $\tau$ for each stratum. For a track $\tau$ that carries $\psi_\beta$, there is a train track map $f_\beta:\tau\to\tau$ induced by $\psi_\beta$, from which we can read fixed point information of $\psi_\beta$ and $\psi_h$. The braid $\beta$ which induces $f_\beta$ is unique up to conjugation and powers of full twists.

On the special track $\tau$ in each stratum, we systematically study all train track maps $f_\beta$ which could potentially be induced by a pseudo-Anosov $\psi_\beta$ that lifts to a FPF map $\psi_h:S\to S$. This analysis is mostly combinatorial in nature, but requires some additional geometric input for $(1;1^5;3^2)$. The end result is a list of exactly three candidate train track maps corresponding (up to conjugation and full twists) to distinct braids $\beta_1,\beta_2, \beta_3$. All three of these candidates are in the stratum $(1;1^5;3^2)$, and we show that no braid in the stratum $(1;1^5;4)$ lifts to a FPF map in the cover. We can then use the requirements that $|c(\beta)|<1$ and $\widehat{\beta}=T(3,5)$ to eliminate the three candidate braids. Technically, this last step requires considering the related braids $\Delta^{2k}\beta_i^{\pm1}$ for any $k$, but a short inspection of fractional Dehn twist coefficients facilitates the argument.

\begin{ack}
A special thank you goes to Ethan Farber and Luya Wang--- this project grew from earlier work with them. Thanks also to Jacob Caudell for inspiring me to look at the Poincar\'e Sphere, and for many helpful discussions. Lastly, I'd like to thank my advisor, John Baldwin, for his constant advice, encouragement, and support.
\end{ack}

\section{Background and setup for the proof of Theorem \ref{thm:FPF}}\label{sec:background}
\subsection{Pseudo-Anosov maps, fibered knots, and Floer theory}

In this subsection, we will review the basic setup to understand most of the Outline (Section \ref{sec:outline}), and to prove Theorem \ref{thm:lspace} from Theorem \ref{thm:FPF}. See the next two subsections \ref{sec:tt} and \ref{sec:lift} for a review of the theory of train track maps and how we will use them in this paper.

\begin{figure}[htp]
    \labellist
    \endlabellist
    \centering
    \includegraphics[width=\textwidth]{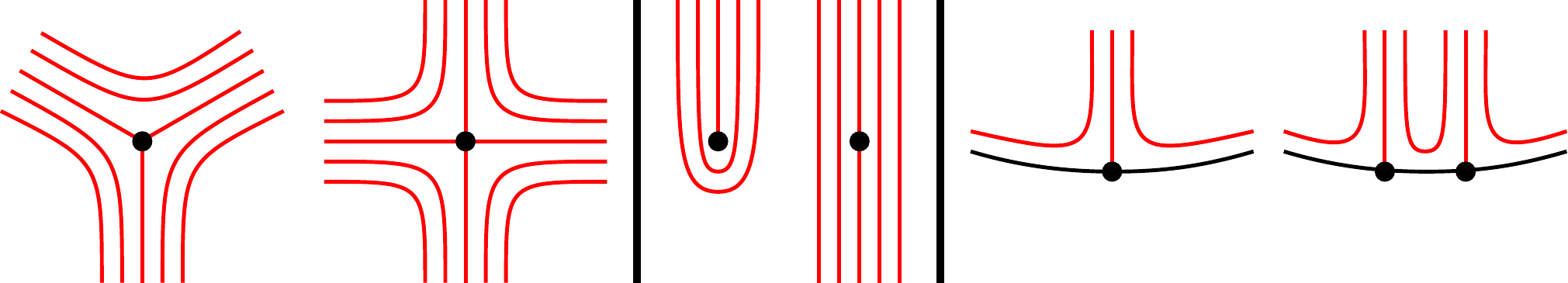}
    \caption{The allowable singularities in the invariant foliations of a pseudo-Anosov map. Non-marked interior singularities must have at least 3 prongs. Left: a 3-pronged and a 4-pronged singularity. Center: a 1-pronged and a 2-pronged singularity at a marked point. Right: a 1-pronged and a 2-pronged boundary component.}
    \label{fig:pAsing}
\end{figure}

\subsubsection{Mapping classes, pseudo-Anosovs, and singularity types}\label{sec:pa}
Let $S$ be a compact surface, possibly with boundary and/or with marked points. The \textit{mapping class group} $\text{Mod}(S)$ is the group of isotopy classes of homeomorphisms $h:S\to S$ which fix $\del S$ pointwise and permute the set of marked points. A \textit{mapping class} is an element of the mapping class group, and an \textit{$n$-braid} is mapping class on the $n$-marked disk.

A map $\psi:S\to S$ is \textit{pseudo-Anosov} if there is a constant $\lambda>1$ and a pair of transverse, measured, singular foliations $(\cal{F}_u,\mu_u)$ and $(\cal{F}_s,\mu_s)$ such that:
\begin{itemize}
    \item $\psi(\cal{F}_u,\mu_u) = (\cal{F}_u,\lambda\mu_u)$
    \item $\psi(\cal{F}_s,\mu_s) = (\cal{F}_s,\lambda^{-1}\mu_s)$
\end{itemize}
and the singularities of $\cal{F}_u$ and $\cal{F}_s$ are as described in Figure \ref{fig:pAsing}. We require that non-marked interior singular points have at least 3 prongs, and that boundary components consist of some number $k\geq 1$ of 1-pronged singularities, which we often think of collectively as a ``$k$-pronged boundary."

Locally, one should think of a pseudo-Anosov map $\psi$ as stretching $S$ in the direction of $\cal{F}_u$ and shrinking $S$ in a transverse direction described by $\cal{F}_s$. The constant $\lambda>1$ is called the \textit{dilatation} of $\psi$, and records how much $\psi$ stretches/shrinks $S$ in each direction. Because the invariant foliations are set-wise preserved by $\psi$, note that $\psi$ always permutes the set of singular points with a given number of prongs.

It will be convenient to refer to the \textit{singularity type} of a pseudo-Anosov map $\psi$, by which we mean the tuple recording the number of prongs at each singularity of $\cal{F}_u$ and $\cal{F}_s$. We will denote the singularity type of $\psi$ by the tuple $(b_1,...,b_r;m_1,...,m_n;k_1,...,k_s)$ where the $i^\text{th}$ boundary component of $\psi$ has $b_i$ prongs; the $i^\text{th}$ marked point has $m_i$ prongs; and the $i^\text{th}$ non-marked interior singularity has $k_i$ prongs. We will use $\emptyset$ if there are no singularities of a certain type, and we will use exponents to denote multiple singularities of the same type with the same number of prongs. For example, the tuple $(3;1^5;\emptyset)$ indicates that $\cal{F}_u$ and $\cal{F}_s$ have a 3-pronged singularity at the unique boundary component of $S$; five 1-pronged singularities at marked points; and no non-marked interior singularities. The tuple $(2^3;2,4;3^2)$ indicates that $\cal{F}_u$ and $\cal{F}_s$ have 2-pronged singularities on each of the three boundary components; a 2-pronged singularity at one marked point, and a 4-pronged singularity at the other; and two 3-pronged singularities at non-marked interior points.

When $S$ has non-empty boundary, pseudo-Anosov maps never fix $\del S$ point-wise, and hence do not represent well-defined mapping classes. Nonetheless, the Nielsen--Thurston classification demonstrates the ubiquity of pseudo-Anosov maps in the study of mapping classes:

\begin{thm}[\cite{Thurston}]\label{thm:NTclass}
Let $S$ be a compact surface, possibly with marked points. Any mapping class $h\in\text{Mod}(S)$ is freely isotopic rel. marked points to a unique homeomorphism $\psi_h:S\to S$ satisfying one of the following:
\begin{itemize}
    \item $\psi_h^n=\text{id}$ for some power $n$.
    \item There is a collection of disjoint simple closed curves $C$ on $S$ for which $\psi_h(C)$ is isotopic to $C$.
    \item $\psi_h$ is pseudo-Anosov
\end{itemize}
\end{thm}

We call $\psi_h$ the \textit{geometric representative} of $h$, and we say that $\psi_h$ is \textit{periodic} in the first case, and \textit{reducible} in the second. These first two cases are not mutually exclusive, but the pseudo-Anosov case does not overlap with either of them. We will also refer to the mapping class $h\in\text{Mod}(S)$ as periodic, reducible, or pseudo-Anosov according to the Nielsen--Thurston type of $\psi_h$.
\begin{rem}
Note the use of \textit{free} isotopy in the theorem statement: when $S$ has non-empty boundary, there are infinitely many mapping classes with the same geometric representative, which are all freely isotopic but non-isotopic rel. boundary. Two mapping classes have the same geometric representative exactly when they differ by a product of Dehn twists about components of $\del S$.
\end{rem}

\subsubsection{Fibered knots and fractional Dehn twists}

When $\del S$ is connected, and for any mapping class $h\in\text{Mod}(S)$, we may associate to $h$ a \textit{fibered knot} $K$ in a 3-manifold $Y$. First, define $Y\cong S\times [0,1]/\sim$ where the relation $\sim$ is defined by:

\begin{itemize}
    \item $(x,0)\sim (h(x),1)$ for all $x\in S$
    \item $(x,s)\sim (x,t)$ for all $x\in \del S$ and all $s,t\in[0,1]$
\end{itemize}
Then, define $K\subset Y$ to be the image of $\del S$ in the quotient. Note that in this construction, the knot exterior $Y_K$ is homeomorphic to the mapping torus $M_h$. Moreover, any copy of $S$ in the quotient $Y$ is naturally a Seifert surface for $K$, of minimal genus. Starting from a knot instead, we say that a knot $K\subset Y$ is \textit{fibered} if its exterior is the mapping torus of a homeomorphism on a Seifert surface.

The pair $(S,h)$ of surface and mapping class is called an \textit{open book decomposition} of the 3-manifold $Y$, or simply an \textit{open book}. Many properties of the open book $(S,h)$ and its geometric representative $\psi_h$ can be interpreted in terms of the fibered knot $K\subset Y$. For example, Thurston's geometrization of fibered 3-manifolds relates the geometry of the knot exterior $Y_K$ to the Nielsen--Thurston type of $\psi_h$:
\begin{thm}[\cite{ThurstonFibered}]\label{thm:geom}
Let $K\subset Y$ be a fibered knot with associated open book $(S,h)$. Then, the exterior $Y_K\cong M_h$ is:
\begin{itemize}
    \item Seifert-fibered exactly when $\psi_h$ is periodic
    \item Toroidal exactly when $\psi_h$ is reducible
    \item Hyperbolic exactly when $\psi_h$ is pseudo-Anosov
\end{itemize}
\end{thm}

\begin{figure}
    \labellist
    \pinlabel $1$ at 11 -2
    \pinlabel $2$ at 52 -5
    \pinlabel $3$ at 94 -2
    \pinlabel $3$ at 188 -2
    \pinlabel $1$ at 229 -5
    \pinlabel $2$ at 270 -2
    \pinlabel $\psi_h$ at 136 50
    \endlabellist
    \centering
    \includegraphics{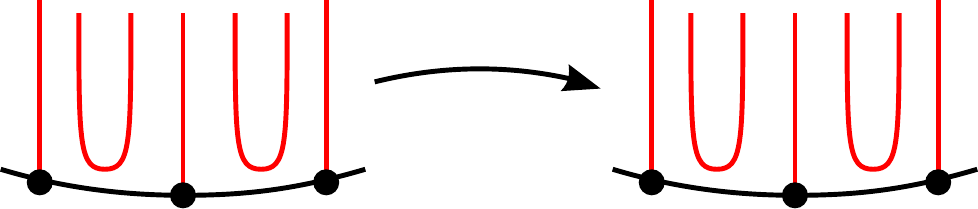}
    \caption{An example where $\psi_h$ acts as a 1/3 rotation on the boundary singular points. The fractional part of $c(h)$ is 1/3.}
    \label{fig:fdtcex}
\end{figure}
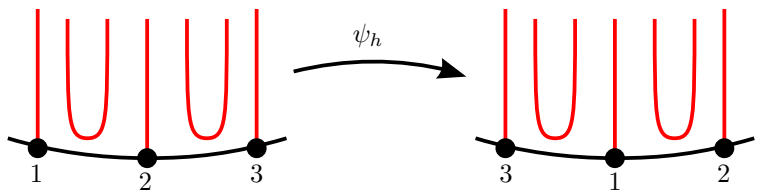

One dynamical aspect of $\psi_h$ which is important to the geometry of $K$ is described by the \textit{fractional Dehn twist coefficient} $c(h)$. Roughly, $c(h)$ measures how the geometric representative $\psi_h$ behaves near $\del S$. When $\psi_h$ is pseudo-Anosov, we define $c(h):=n+m/k$, where $h$ acts as $n$ full twists near $\del S$; the invariant foliations of $\psi_h$ have $k$ singular points on $\del S$; and $\psi_h$ acts as a $m/k$ rotation on the cyclically-ordered set of singular points on $\del S$. An example is shown in Figure \ref{fig:fdtcex}.

Note that when the invariant foliations of $\psi_h$ have a single prong on $\del S$, we have $c(h)\in\bb{Z}$, since there is no ``fractional part." Conversely, if $c(h)\in\bb{Z}$ then we can conclude that $\psi_h$ does not rotate the boundary singularity. Here are some well-known properties of fractional Dehn twist coefficients:

\begin{thm}[\cite{ik}\cite{hkm}]\label{thm:fdtc}
Let $S$ be a compact surface, possibly with marked points, and let $h\in\text{Mod}(S)$ be a mapping class.
\begin{itemize}
    \item $c(h)$ is preserved under conjugation within $\text{Mod}(S)$
    \item $c(\text{id})=0$ and $c(D_{\del S}^n\circ h^k)=n+kc(h)$
    \item If $h$ admits a factorization into positive Dehn twists and half twists, then $c(h)\geq 0$. If $h$ is also pseudo-Anosov, then $c(h)>0$.
    \item If $S$ has no marked points and the open book $(S,h)$ supports a tight contact structure, then $c(h)\geq 0$. If $h$ is also pseudo-Anosov, then $c(h)>0$.
\end{itemize}
\end{thm}

Given a fibered knot $K$, we can define $c(K):=c(h)$ where $(S,h)$ is the open book associated to $K$. The first property above implies that this definition is well-defined, because isotopy of $K$ amounts to isotopy and/or conjugation of $h$.

\subsubsection{Floer-theoretic input, and Theorem \ref{thm:lspace} from Theorem \ref{thm:FPF}}\label{sec:floer}

Heegaard Floer homology, defined by Ozsv\'ath--Szab\'o in \cite{OSinvt}, is an invariant $\widehat{\mathit{HF}}(Y)$ of a closed, oriented 3-manifold $Y$, and it takes the form of a graded vector space over $\ff=\zz/2\zz$. The vector space decomposes over $\spinc$-structures on $Y$, and satisfies $\chi(\widehat{\mathit{HF}}(Y))=|H_1(Y;\zz)|$ when $Y$ is a rational homology $S^3$. In particular, we have $\dim \widehat{\mathit{HF}}(Y)\geq |H_1(Y;\zz)|$ for any rational homology sphere $Y$. When equality is achieved, we call $Y$ an \textit{L-space}. For some concrete examples, $S^3$, $\cal{P}$, and Lens spaces are all L-spaces.

A knot $K\subset Y$ which admits a non-trivial surgery to an L-space is called an \textit{L-space knot}, and in this case the geometry of $K$ is quite constrained. For example, combining results of \cite{Tange}, \cite{OSknot}, \cite{Nifibered}, and \cite{Ni2}, we have:

\begin{thm}\label{thm:hfk}
Suppose that $Y$ is an integer homology sphere L-space, and let $K\subset Y$ be an L-space knot with irreducible exterior. Then, we have:
\begin{itemize}
    \item Every coefficient of the Alexander polynomial of $K$ is 0 or $\pm1$, the non-zero coefficients alternate in sign, and the top two coefficients are non-zero
    \item $K$ is fibered and the open book decomposition $(S,h)$ for $K$ supports a tight contact structure
    \item If $K$ is hyperbolic, then $\psi_h$ has no interior fixed points
\end{itemize}
\end{thm}

We can now see how Theorem \ref{thm:lspace} follows quickly from Theorem \ref{thm:FPF}.

\begin{proof}[proof of Theorem \ref{thm:lspace} from Theorem \ref{thm:FPF}]

Let $K\neq\cal{K}$ be a genus-two L-space knot in $\cal{P}$ with irreducible exterior. Our goal is to show that $K$ is a hyperolic, fibered, FPF knot with $c(K)\neq 0$. By Theorem \ref{thm:hfk} we already know that $K$ is fibered, and if $K$ is hyperbolic then $K$ is FPF. Let $(S,h)$ be the open book associated to $K$, as a fibered knot. If $K$ is hyperbolic, then $h\in\text{Mod}(S)$ is a pseudo-Anosov mapping class by Theorem \ref{thm:geom}. We also know that $(S,h)$ supports a tight contact structure by Theorem \ref{thm:hfk}. So, if $K$ is hyperbolic then $c(K)=c(h)>0$ by Theorem \ref{thm:fdtc}.

It thus suffices to show that $K$ is hyperbolic, because that will imply that $K$ is FPF and that $c(K)\neq 0$. Any irreducible knot exterior in $\cal{P}$ is hyperbolic, Seifert fibered, or toroidal. And, $\cal{K}$ is the unique genus-two knot in $\cal{P}$ with Seifert-fibered exterior. So, $K\neq\cal{K}$ has either hyperbolic or toroidal exterior.

Finally, suppose $K$ has toroidal exterior. It follows from Theorem \ref{thm:geom} that the mapping class $h\in\text{Mod}(S)$ is reducible. Thinking of the Alexander polynomial $\Delta_K(t)$ of $K$ as the characteristic polynomial of the action $h_\ast:H_1(S)\to H_1(S)$, we can see that $\Delta_K(t)$ is reducible (because the action of $h_\ast$ fixes a nontrivial subspace corresponding to the fixed mutlicurve of $\psi_h$). However, one can quickly compute using the constraints of Theorem \ref{thm:hfk} that $\Delta_K(t)=t^4-t^3+t^2-t+1$, which is irreducible.
\end{proof}

\subsection{Train tracks, train track maps, and standard tracks}\label{sec:tt}

\begin{figure}[t]
    \centering
    \includegraphics[width=\textwidth]{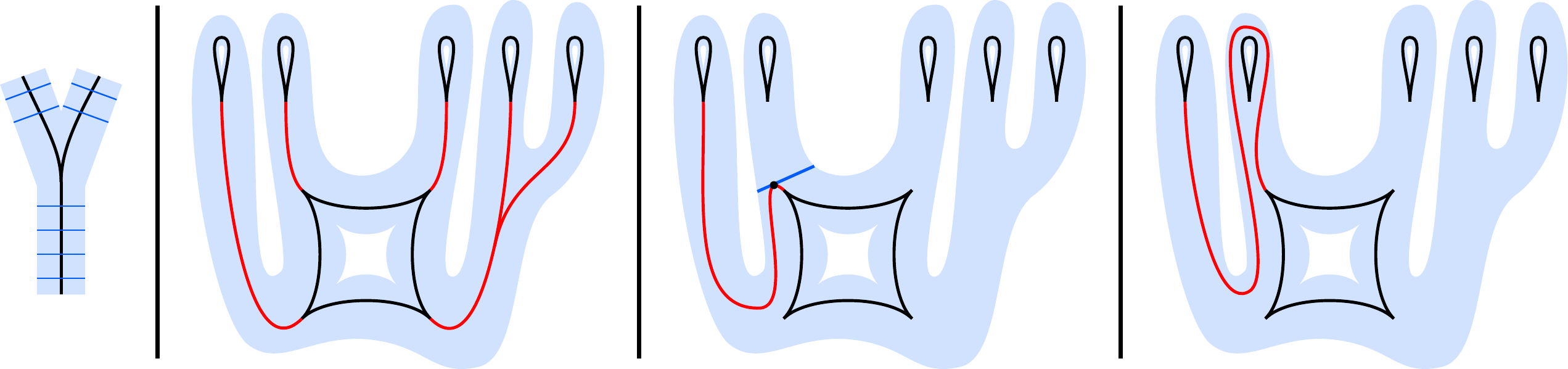}
    \caption{Left: leaves of a fibered neighborhood near a switch. Center left: a fibered neighborhood of a track $\tau$. Center right: the image of the red edge is not transverse to the leaves. Right: the image of the red edge is transverse to the leaves, and ready to be collapsed onto $\tau$.}
    \label{fig:fiberednbd}
\end{figure}

\subsubsection{Train tracks carrying mapping classes}
Our eventual goal will be to classify fixed-point-free pseudo-Anosov maps in certain strata on the genus-two surface with one boundary component. To do that, we will use the theory of train tracks, which may be thought of as a combinatorial perspective on the geometry of mapping classes.

A \textit{train track} $\tau\subset S$ is a branched 1-manifold embedded in the interior of a compact surface $S$. We will refer to the non-manifold points of $\tau$ as \textit{switches}. Thinking of $\tau$ as a graph whose vertices are its switches, there is a natural cyclic ordering on the edges adjacent to any given switch. We say that two edges of $\tau$ form a \textit{cusp} if they are adjacent at a switch; share a common tangent direction; and are adjacent in the natural cyclic ordering.

A \textit{train path} on a train track $\tau$ is a smooth immersion $\gamma:[0,1]\to\tau$ such that $\gamma(0)$ and $\gamma(1)$ are switches. For example, any edge of $\tau$ is a train path. A \textit{train track map} $f:\tau\to\tau$ is a surjection such that for any train path $\gamma:[0,1]\to\tau$, the composition $f\circ\gamma$ is also a train path. Fix an enumeration $e_1,...,e_n$ of the edges of $\tau$; the \textit{transition matrix} $M(f)$ is the matrix whose $(i,j)$ entry is the number of times that the train path $f(e_i)$ passes through $e_j$.

A \textit{fibered neighborhood} $F(\tau)\subset S$ is a tubular neighborhood of $\tau$ which is foliated by intervals transverse to $\tau$, away from the switches. The intervals are called \textit{leaves}. A geometric map $\psi_h:S\to S$ is \textit{carried by} $\tau$ if $\psi_h(\tau)$ can be isotoped into $F(\tau)$ so that the image is everywhere transverse to the leaves of the foliation. We say that a mapping class $h\in\text{Mod}(S)$ is carried by $\tau$ if its geometric representative is. See Figure \ref{fig:fiberednbd} for some local pictures, and then look at Figure \ref{fig:4tracklift} for an image of the entirety of $\psi_h(\tau)\subset F(\tau)$.

When a mapping class $h$ is carried by $\tau$, it induces a train track map $f_h:\tau\to\tau$, determined by the image $\psi_h(\tau)$ under the geometric representative. For an edge $e$ of $\tau$, we define $f_h(e)$ to be the train path that the image $\psi_h(e)\subset F(\tau)$ collapses onto under the natural deformation retraction $F(\tau)\to\tau$. The induced train track map $f_h$ completely determines $\psi_h$:

\begin{thm}[\cite{BH}]\label{thm:ttMCG}
Let $h,g\in\text{Mod}(S)$ be two mapping classes with geometric representatives $\psi_h,\psi_g$. Suppose that $h$ and $g$ are carried by the same train track $\tau$, and induce maps $f_h,f_g:\tau\to\tau$. If $f_h=f_g$, then $h$ and $g$ are freely isotopic and $\psi_h=\psi_g$.
\end{thm}

\begin{cor}
Let $\alpha,\beta\in\text{Mod}(D_n)$ be two $n$-braids which are carried by the same track $\tau$ and induce the same train track map $f_\alpha=f_\beta:\tau\to\tau$. Then, we have $\psi_\alpha=\psi_\beta$ and $\alpha=\Delta^{2k}\beta$ for some $k\in\bb{Z}$, where $\Delta^2=(\sigma_1...\sigma_{n-1})^n$ is a full twist about $\del D_n$.
\end{cor}

Given a mapping class $h\in\text{Mod}(S)$, we define the \textit{geometric data} of $h$ to be a tuple $(h, \psi_h,\tau,f_h)$ where $\psi_h$ is the geometric representative of $h$; $\tau$ carries $h$; and $f_h:\tau\to\tau$ is the induced train track map. It should be noted that any mapping class $h$ is always carried by many different train tracks. But, if $\tau$ is a train track which carries $h$, then the induced train track map $f_h$ is well-defined.

Train tracks which carry mapping classes, and train track maps which are induced by mapping classes, are somewhat special. There are many train tracks which do not carry any mapping classes. And, on any given train track, it is easy to produce train track maps which are not induced by any mapping class. Indeed, the geometry of a train track is determined by the geometry of. For example, when the mapping class is pseudo-Anosov, we have:

\begin{prop}[\cite{BH}]
Let $h$ be a pseudo-Anosov mapping class carried by $\tau$. If $N$ is a connected interior (resp. peripheral) component of $S\setminus\tau$, then $\tau$ has $p$ cusps along $\del N$ if and only if $\psi_h$ has a $p$-pronged singularity in the interior of $N$ (resp. a $p$-pronged boundary in $N$).
\end{prop}

For example, the train track in Figure \ref{fig:fiberednbd} carries maps in the stratum $(1;1^5;4)$: the peripheral region has 1 cusp; each marked point lies in a monogon region; and there is a unique non-marked interior region, which has 4 cusps. We will often refer to a train track as belonging to a certain stratum, by which we mean it carries maps in that stratum.

When a train track map is induced by a mapping class, much of the interesting geometric information of the mapping class can be read from the track track map. For example:

\begin{prop}[\cite{BH}\cite{Los}\cite{CC}]\label{prop:fix}
Let $(h,\psi_h,\tau,f_h)$ be the geometric data of a mapping class. Then, the transition matrix $M(f_h)$ is Peron--Frobenius iff the mapping class is pseudo-Anosov.
In this case, the number of fixed points of $\psi_h$ is bounded in terms of the trace of $M(f_h)$: $$\frac{1}{2}\text{tr}(M(f_h))\leq|\text{Fix}(\psi_h)|$$
\end{prop}

\subsubsection{Real and infinitesimal edges, and the Bestvina--Handel construction}
Given a geometric map $\psi_h$, Bestvina--Handel in \cite{BH} construct a train track $\tau$ which carries $\psi_h$. The key object of the Bestvina--Handel construction is an \textit{efficient fibered surface} associated to $\psi_h$. The fibered surface $F\subset S$ is a deformation retract of $S$ which is decomposed into \textit{strips} and \textit{junctions} determined by $\psi_h$. Roughly, the junctions are disks which record the periodic pieces of $S$ under the action of $\psi_h$, and the strips are rectangles foliated by lines parallel to the junctions.

Given a fibered surface $F$ for a geometric $\psi_h$, we obtain a graph $G$ by collapsing each junction to a vertex and each strip to an edge. Bestvina--Handel explain how to insert \textit{infinitesimal} edges into each junction according to the geometry of $\psi_h$, so that collapsing strips to edges and junctions to their infinitesimal edges forms a train track $\tau$ which carries $\psi_h$. The edges of $\tau$ formed by strips are called $\textit{real}$ edges.

The structure of the real and infinitesimal edges of $\tau$ is crucial to the fidelity of Bestvina--Handel's construction. As the infinitesimal edges were inserted into periodic pieces of the fibered surface, it follows that $\psi_h$ sends infinitesimal edges to infinitesimal edges. In particular, the train track map $f_h:\tau\to\tau$ induced by $\psi_h$ sends infinitesimal edges to infinitesimal edges.

Moreover, the cusp structure of the infinitesimal edges in a given junction is preserved under $\psi_h$. An \textit{infinitesimal polygon} is a connected component of $S\setminus\tau$ whose boundary is a union of infinitesimal edges. By the Bestvina--Handel construction, $\psi_h$ and $f_h$ send infinitesimal polygons of $\tau$ to other infinitesimal polygons with the same number of cusps.

For the rest of the paper, we will restrict our attention to train tracks which carry pseudo-Anosov maps. And, we will implicitly presuppose the structure of real and infinitesimal edges of $\tau$, as imposed by an arbitrary map carried by $\tau$. The infinitesimal edges will always be drawn in black, and the real edges in color. For example, in the tracks in Figures \ref{fig:fiberednbd} and \ref{fig:jointstandard}, the real edges are shown in red and the infinitesimal edges in black.

\subsubsection{Standard tracks, jointless tracks, and a carrying theorem}

\begin{figure}
    \centering
    \includegraphics[width=\textwidth]{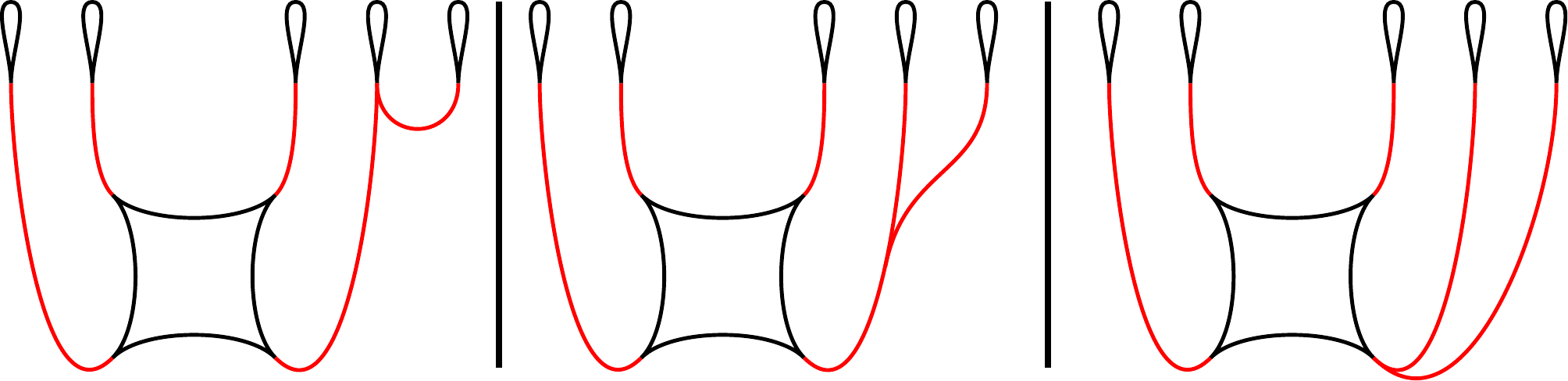}
    \caption{Some train tracks in the stratum $(1;1^5;4)$. Left: a standard track with a joint. Center: a non-standard track (with no joints). Right: a standard, jointless track.}
    \label{fig:jointstandard}
\end{figure}

We say that a train track $\tau\subset D_n$ on an $n$-marked disk $D_n$ is \textit{standard} if:
\begin{itemize}
    \item Every non-peripheral component of $S\setminus\tau$ is an infinitesimal polygon
    \item The switches of $\tau$ are precisely the vertices of the infinitesimal polygons
    \item At each switch, all adjacent infinitesimal edges are tangent to each other; all adjacent real edges are tangent to each other; and no real edge is tangent to an infinitesimal edge
    \item For each marked point $p$ of $D_n$, there is a single edge of an infinitesimal polygon vertically above $p$, and no other edge of $\tau$ is vertically above $p$.
\end{itemize}

See Figure \ref{fig:jointstandard} for an example of a standard train track. The idea of standard train tracks originally appeared in \cite{KLS} (and later \cite{CH} and \cite{HS}) as a tool to study dilatations of pseudo-Anosov braids. In their terminology, a track satisfying the first three properties is \textit{standardly embedded} and a track satisfying the third is in \textit{standard position}.

\begin{prop}[\cite{KLS}]
Every pseudo-Anosov map on a marked disk is carried by a standard train track.
\end{prop}

When the singularity type of a pseudo-Anosov has at least one interior singularity, we can improve this enumeration further. A \textit{joint} of a standard train track $\tau$ is a switch of an infinitesimal monogon surrounding a marked point, which has more than one adjacent real edge. See Figure \ref{fig:jointstandard}. In work with Farber and Wang \cite{FRW}, we proved:

\begin{thm}\label{thm:jointless}
Let $\psi$ be a pseudo-Anosov map on a marked disk. If $\psi$ has at least one interior singularity, then $\psi$ is carried by a standard train track with no joints.
\end{thm}

\subsection{Branched covers and lifting train track maps}\label{sec:lift}

\subsubsection{Braids and the Birman--Hilden correspondence}
We will be interested in studying fibered knots and mapping classes under branched coverings. The Birman--Hilden correspondence describes the relevant operation in our setting. We may view the $n$-marked disk $D_n$ as the quotient of a surface $S$ via the action of a fixed hyperelliptic involution $\iota: S\to S$. When a map $h:S\to S$ commutes with $\iota$, it projects to an $n$-braid $\beta:D_n\to D_n$. In this case, we say that $h$ is \textit{symmetric}. Every $n$-braid lifts to a symmetric map on a surface $S$ of genus $\lfloor\frac{n-1}{2}\rfloor$ with 1 boundary component if $n$ is odd, and 2 boundary components if $n$ is even.

Here is how to interpret the correspondence from the perspective of knots and 3-manifolds. Suppose $\beta$ lifts to $h$, where $(S,h)$ is an open book decomposition for $Y$ describing the fibered knot $K\subset Y$. Then, $Y$ is the double cover over $S^3$ branched along the braid closure $\widehat{\beta}$. And, $K$ is the lift to $Y$ of the unknotted braid axis for $\beta$ in $S^3$ under this covering.

The Nielsen--Thurston class is preserved under the Birman--Hilden correspondence, i.e. $\psi_\beta$ is pseudo-Anosov (resp. periodic, reducible) iff $\psi_h$ is (resp. periodic, reducible). In the case that $\psi_h$ and $\psi_\beta$ are pseudo-Anosov, the invariant foliations can be tracked through the lift, as well. Specifically, the number of prongs of the foliations at the marked points double in the lift. For example, 1-pronged marked points of $\psi_\beta$ lift to regular points of $\psi_h$; and 2-pronged marked points of $\psi_\beta$ lift to 4-pronged marked points of $\psi_h$. The number of boundary prongs double in a similar manner. Away from the boundary, however, each non-marked singularity of $\psi_\beta$ lifts to two singularities of $\psi_h$ with the same number of prongs. So, for example, $\psi_\beta$ has singularity type $(1;1^5;4)$ exactly when $\psi_h$ has singularity type $(2;\emptyset;4^2)$. Analyzing how $\psi_\beta$ and $\psi_h$ twist near the boundary leads also to a relationship between fractional Dehn twist coefficients: we always have $c(\beta)=2c(h)$.

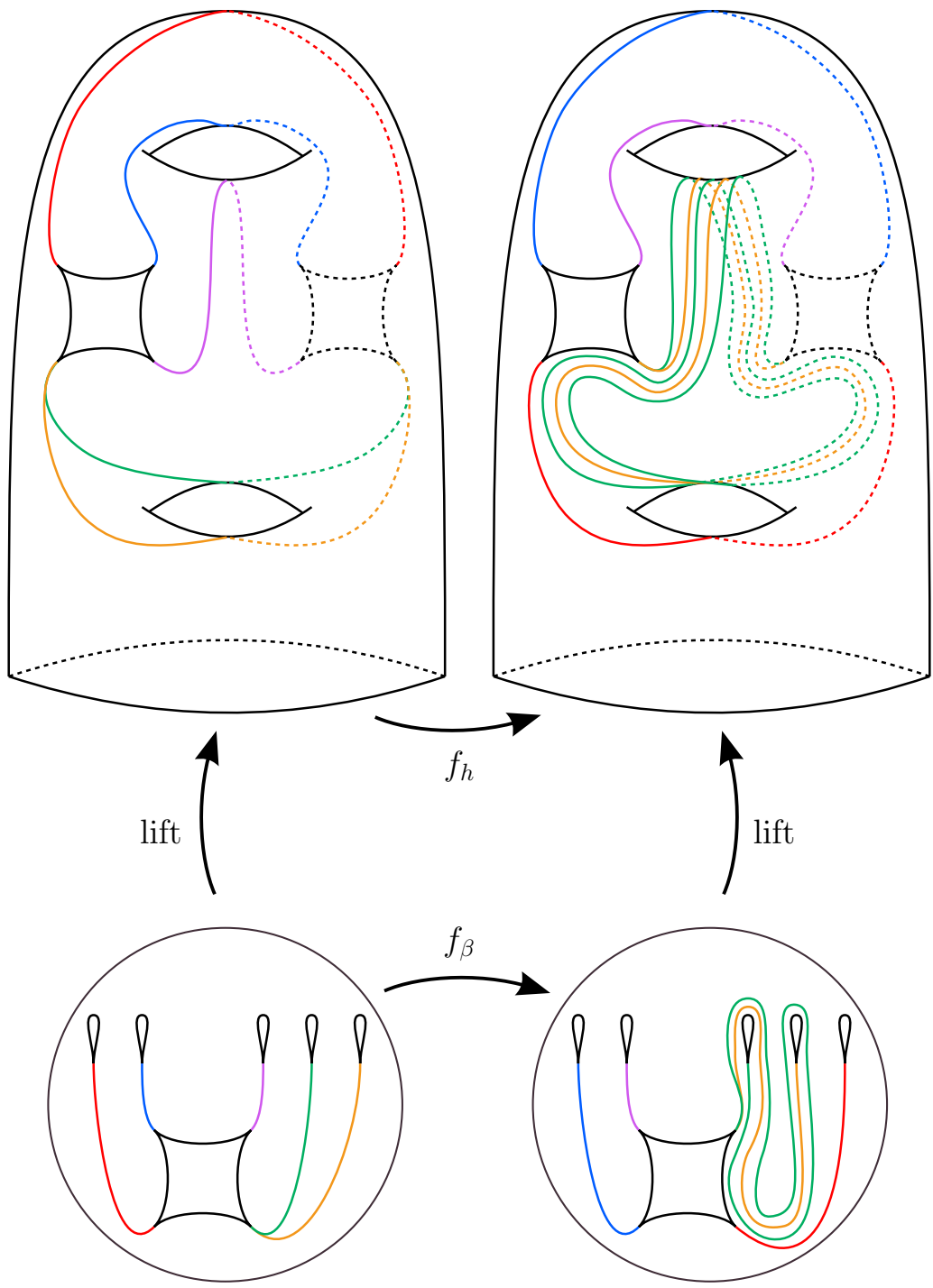
\begin{figure}[htp]
    \labellist
    \Large
    \pinlabel $f_h$ at 280 320
    \pinlabel $\text{lift}$ at 475 280
    \pinlabel $\text{lift}$ at 95 280
    \pinlabel $f_\beta$ at 280 210
    \endlabellist
    \centering
    \includegraphics[width=0.9\textwidth]{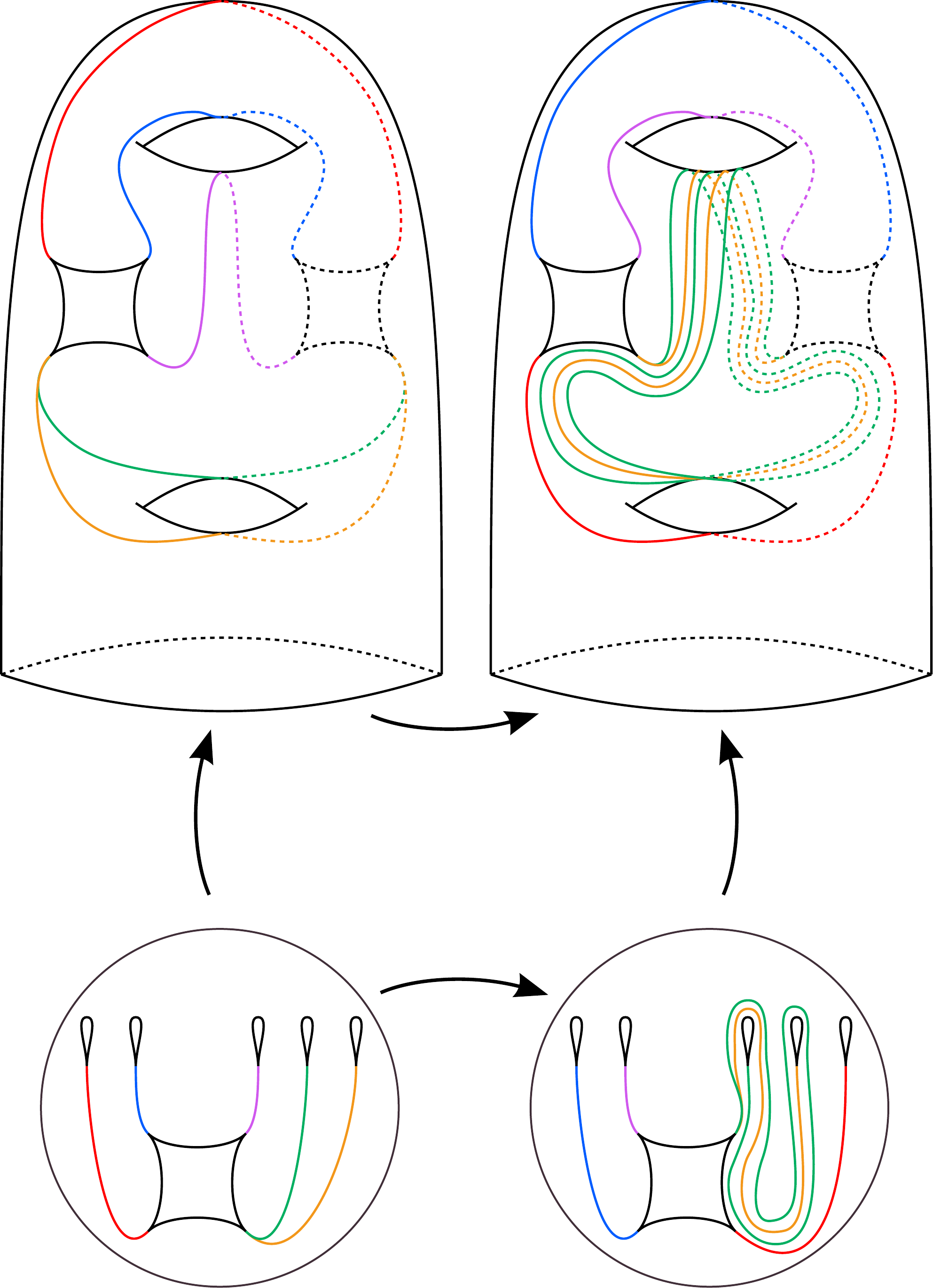}
    \caption{Lifting a train track and train track map from $D_5$ to $S$. Dashed edges indicate the back side of the surface.}
    \label{fig:4tracklift}
\end{figure}

\subsubsection{Lifting standard tracks, and the trace lemma}

\begin{figure}[htp]
    \centering
    \labellist
    \Large
    \pinlabel $a_1$ at 4 90
    \pinlabel $a_2$ at 134 90
    \pinlabel $x_1$ at 30 102
    \pinlabel $x_2$ at 108 102
    \pinlabel $a_1$ at 210 108
    \pinlabel $a_2$ at 277 108
    \pinlabel $a_1$ at 24 70
    \pinlabel $b_2$ at 116 70
    \pinlabel $y_1$ at 43 40
    \pinlabel $y_2$ at 98 40
    \pinlabel $x_1$ at 56 60
    \pinlabel $x_2$ at 80 60
    \pinlabel $b_1$ at 24 6
    \pinlabel $a_2$ at 116 6
    \pinlabel $a_1$ at 204 68
    \pinlabel $b_2$ at 284 68
    \pinlabel $y_1$ at 216 40
    \pinlabel $y_2$ at 272 40
    \pinlabel $x_2$ at 244 66
    \pinlabel $x_1$ at 244 10
    \pinlabel $b_1$ at 204 8
    \pinlabel $a_2$ at 284 8
    \endlabellist
    \includegraphics[width=0.95\textwidth]{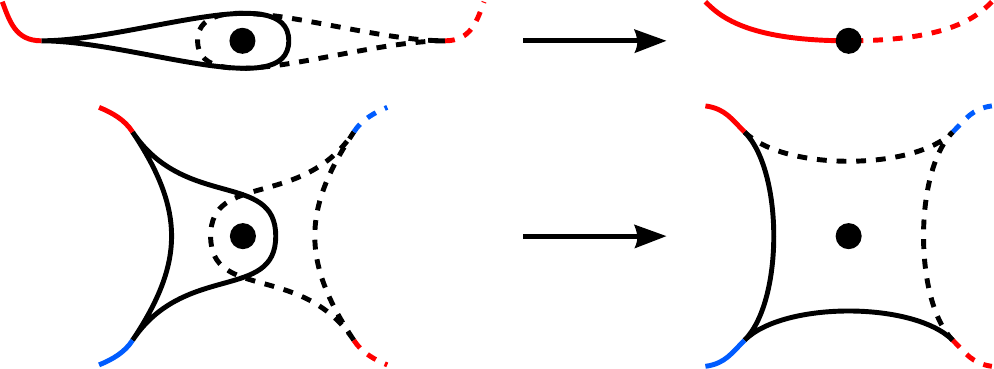}
    \caption{Smoothing infinitesimal polygons near a marked point. Top: smoothing monogons to a regular point. Bottom: smoothing bigons to a rectangle. In both images, the involution $\iota$ is a $180^\circ$ rotation about the marked point, and the side-swapping edge in $\tau$ is $x$, which lifts to $x_1,x_2$ in $\widetilde{\tau}$.}
    \label{fig:liftpolygon}
\end{figure}

Let $(\beta,\psi_\beta,\tau,f_\beta)$ be the geometric data of a pseudo-Anosov $n$-braid. We will explicitly construct related data $(h,\psi_h,\widetilde{\tau},f_h)$ for the lift $h$ of $\beta$ under the Birman--Hilden correspondence. For simplicity, assume that $\tau$ is standard (though a similar construction will work for non-standard tracks), and start by lifting $\tau$ to obtain two copies of $\tau$ on $S$. Because we assumed $\tau$ to be standard, the two copies of $\tau$ in the lift together produces a train track away from the marked points. But, near the marked points, we see two copies of the infinitesimal polygons surrounding each marked point, which we need to paste together coherently. To do this, simply replace each pair of infinitesimal $p$-gons surrounding a marked point with a single infinitesimal $2p$-gon, as in Figure \ref{fig:liftpolygon}. If $p=1$, instead replace the pair of monogons with a smooth point. See Figure \ref{fig:4tracklift} for an example of a full lifted track $\widetilde{\tau}$.

This construction produces a train track $\widetilde{\tau}$ for which edges come in pairs $e_1,e_2$ of lifts of a single real edge $e$ of $\tau$; the pairs satisfy $\iota(e_1)=e_2$. Moreover, any train path $e(1)...e(n)$ in $\tau$ always lifts to exactly two train paths $e(1)_1...e(n)_{n_i}$ and $e(1)_2...e(n)_{n_j}$ in $\widetilde{\tau}$. These two train paths are distinguished by which ``side" of $\widetilde{\tau}$ the paths start on. Now, for an arbitrary real edge $e$ of $\tau$, let $f_\beta(e)=e(1)...e(n)$. To construct $f_h:\widetilde{\tau}\to\widetilde{\tau}$, simply define $f_h(e_i)=e(1)_{i_1}...e(n)_{i_n}$ to be the unique lift of $f_\beta(e)$ which starts on the same side of $\widetilde{\tau}$ as $\psi_h(e_i)$ does. An example of a lifted train track map is shown in Figure \ref{fig:4tracklift}.

The explicit construction of $f_h$ from $f_\beta$ allows us to examine fixed point properties of $\psi_h$ in terms of the transition matrix $M(f_\beta)$. Recall that a standard track on the marked disk has a unique edge of an infinitesimal polygon above each marked point. For the statement and proofs of the next two lemmas we will refer to these edges as \textit{side-swapping edges}: their key feature is that when a train path runs over a side-swapping edge, it swaps to the ``other side" of $\widetilde{\tau}$. See Figure \ref{fig:liftpolygon}. The following lemma will be crucial for our analysis.

\begin{lem}[The trace lemma]\label{lem:trace}
Let $(\beta,\psi_\beta,\tau,f_\beta)$ be the data of a pseudo-Anosov braid with lift $(h,\psi_h,\widetilde{\tau},f_h)$ on $S$. Suppose that $\tau$ is standard and that $\psi_h$ has no interior fixed points on $S$. Then, for any real edge $e$ of $\tau$, and between any two occurrences of $e$ in $f_\beta(e)$, the image path $f_\beta(e)$ must contain an even number of side-swapping edges.
\end{lem}
\begin{proof}
Because $\psi_h$ has no interior fixed points, we know $M(f_h)$ must be traceless, by Proposition \ref{prop:fix}. In particular, for any real edge $e_i\in\widetilde{\tau}$, we know that $e_i$ does not appear in $f_h(e_i)$.

Now, note that every time $f_\beta(e)$ passes over a side-swapping edge, the image $f_h(e_i)$ of either lift of $e$ crosses over to the other side of $\widetilde{\tau}$. So, if $e$ appears in $f_\beta(e)$ with an odd number of side-swapping edges between adjacent appearances, then both lifts $e_1$ and $e_2$ appear in the image $f_h(e_1)$ (and same for $f_h(e_2)$).
\end{proof}

The following is a special case of the trace lemma as stated above, and is identical to what was originally referred to as the trace lemma in \cite{FRW}. Because the train tracks relevant to this paper will always have jointless monogons at all marked points, it will be a very useful simplification.

\begin{cor}[The trace lemma for jointless monogons]\label{lem:tracemon}
Retain the same notation and assumptions as in the previous lemma. If $e$ is a real edge of $\tau$ which is incident to a jointless monogon at a marked point, then $e$ does not appear in $f_\beta(e)$.
\end{cor}
\begin{proof}
Suppose that $e$ does appear in $f_\beta(e)$. Because $e$ is adjacent to a jointless monogon, either there are two adjacent appearances of $e$, or $f_\beta(e)$ ends on $e$. The first case is not allowed by the previous lemma, and the second case produces a fixed point in the lift, at the end of $f_h(e_i)$.
\end{proof}

\section{The strata $(4;\emptyset;3^2)$, $(6;\emptyset;\emptyset)$, and $(2;\emptyset;4^2)$}\label{sec:244}

The rest of the paper will be devoted to proving Theorem \ref{thm:FPF}, as explained in the outline (Section \ref{sec:outline}). From the argument in the outline, it suffices to consider pseudo-Anosov maps in just four strata: $(6;\emptyset;\emptyset)$, $(4;\emptyset;3^2)$, $(2;\emptyset;4^2)$ and $(2;\emptyset;3^4)$. The first three of these strata will be relatively easy to deal with, so we will study them all in this section. Then, we will examine the stratum $(2;\emptyset;3^4)$ in Section \ref{sec:23333}, which will take considerably more work.

\subsection{The strata $(4;\emptyset;3^2)$ and $(6;\emptyset;\emptyset)$}\label{sec:6433}

We can initiate the proof of Theorem \ref{thm:FPF} by completing our analysis of the strata $(4;\emptyset;3^2)$ and $(6;\emptyset;\emptyset)$. These strata were first studied in \cite{FRW}, so we will be able to deal with them quickly and then move onto the new strata.

\begin{subthm}\label{thm:433}
Let $h\in\Mod(S)$ be a pseudo-Anosov mapping class with geometric representative $\psi_h$. Suppose that $\psi_h$ is FPF and has singularity type $(3;\emptyset;4^2)$. Then, $(S,h)$ is not an open book decomposition for $\cal{P}$.
\end{subthm}
\begin{proof}
By \cite{FRW}, we know that such an $h$ is conjugate to the lift of $\Delta^{2k}\beta_n^{\pm1}$ for some $k\in\bb{Z}$ and $n\geq 0$, where $\beta_n=\sigma_1^{n+2}\sigma_2\sigma_3\sigma_4\sigma_1\sigma_2\sigma_3\sigma_4^2$. We may compute that, as a braid, $c(\Delta^{2k}\beta_n^{\pm1})=k\pm\frac{1}{2}$. Let $\beta$ be the braid from this class that lifts to our hypothetical $h$. Because $\cal{P}$ is an L-space, we then know that $|c(\beta)|<2$ by \cite{HM}. So, we can conclude that $\beta$ must be conjugate to one of $\beta_n^{\pm1}$, $\Delta^{\pm2}\beta_n^{\mp1}$, $\Delta^{\pm2}\beta_n^{\pm1}$, or $\Delta^{\pm4}\beta_n^{\mp1}$. Moreover, because $(S,h)$ is an open book decomposition for $\cal{P}$, we know that, as a knot, $\widehat{\beta}=T(3,5)$. For any braid $\beta$, the closure $\widehat{\beta^{-1}}$ is the mirror of $\widehat{\beta}$, so it suffices to check that none of $\beta_n$, $\Delta^2\beta_n^{-1}$, $\Delta^2\beta_n$, or $\Delta^4\beta_n^{-1}$ close up to $T(3,5)$. 

By \cite{FRW}, we know that $\widehat{\beta_n}=T(2,n+7)$ and $\widehat{\Delta^2\beta_n^{-1}}=P(3,3-n,-2)$, none of which are isotopic to $T(3,5)$. For $\Delta^2\beta_n$, we may simply compute the self-linking number: $sl(\Delta^2\beta_n)=25+n$ is greater than the maximal self-linking number $\overline{sl}(T(3,5))=7$ of $T(3,5)$. For $\Delta^2\beta_n^{-1}$, simply note that $\det \widehat{\Delta^4\beta_n^{-1}}\neq 1$: if the determinant were 1, then $\det\widehat{\beta_n^{-1}}=1$, too, but instead $\det\widehat{\beta_n^{-1}}=\det T(2,n+7)=n+7>1$.
\end{proof}

\begin{subthm}\label{thm:6}
Let $h\in\Mod(S)$ be a pseudo-Anosov mapping class with geometric representative $\psi_h$. Suppose that $\psi_h$ is FPF and has singularity type $(6;\emptyset;\emptyset)$. Then, $(S,h)$ is not an open book decomposition for $\cal{P}$.
\end{subthm}
\begin{proof}
By \cite{FRW}, we know that the geometric representative $\psi_h$ of such an $h$ achieves the minimal dilatation $\lambda(\psi_h)=\lambda_2$ in genus-two. Following a similar argument to the proof of Proposition 3.7 from \cite{FRW}, we can see that $h$ is the lift of $\Delta^{2k}\alpha^{\pm1}$ for some $k\in\bb{Z}$, where $\alpha=\sigma_1\sigma_2\sigma_3\sigma_4\sigma_1\sigma_2$. Here, $\alpha$ was originally identified by \cite{HS} as the dilatation-minimizing 5-braid for the stratum $(3;1^5;\emptyset)$ on the 5-marked disk, which lifts to the stratum $(6;\emptyset;\emptyset)$ on $S$.

Using an argument as in the proof of the previous proposition, it suffices to check that none of the braids $\alpha$, $\Delta^2\alpha^{-1}$, $\Delta^2\alpha$, or $\Delta^4\alpha^{-1}$ close up to $T(3,5)$. For all except $\alpha$, one can use a self-linking number computation to confirm this, and one can easily perform an isotopy to see that $\widehat{\alpha}=T(2,3)$.
\end{proof}

\subsection{The stratum $(2;\emptyset;4^2)$}

This stratum has not yet been analyzed, so we will need to perform some train track map analysis. Our goal for this subsection is to prove:
\begin{subthm}\label{thm:244}
Every pseudo-Anosov map in the stratum $(2;\emptyset;4^2)$ has an interior fixed point.
\end{subthm}

By the discussion in the outline (subsection \ref{sec:outline}), if $h\in\text{Mod}(S)$ is a pseudo-Anosov mapping class with FPF geometric representative $\psi_h$, then $h$ is the lift of a 5-braid $\beta$. If $\psi_h$ has singularity type $(2;\emptyset;4^2)$ then $\psi_\beta$ has singularity type $(1;1^5;4)$, by the discussion in Section \ref{sec:lift}. So, Theorem \ref{thm:244} reduces to the following:

\begin{prop}
Let $\beta$ be a pseudo-Anosov 5-braid whose geometric representative $\psi_\beta$ has singularity type $(1;1^5;4)$. Then, the lift $\psi_h:S\to S$ has an interior fixed point.
\end{prop}
\begin{rem}
Note that for this stratum, we will actually prove that there are no FPF maps, whereas there are FPF maps in the other three strata.
\end{rem}
The proposition will follow from an analysis of train track maps on a special train track in the stratum $(1;1^5;4)$. In this stratum, there is a unique jointless standard train track (up to isotopy): the ``Jellyfish" track, shown in Figure \ref{fig:jellyfishtrackex}. So, by Theorem \ref{thm:jointless}, it suffices to check braids carried by this canonical track.

\begin{figure}
    \labellist
    \pinlabel $r$ at 0 60
    \pinlabel $b$ at 25 110
    \pinlabel $p$ at 100 110
    \pinlabel $g$ at 125 60
    \pinlabel $y$ at 167 60
    \endlabellist
    \centering
    \includegraphics[width=\textwidth]{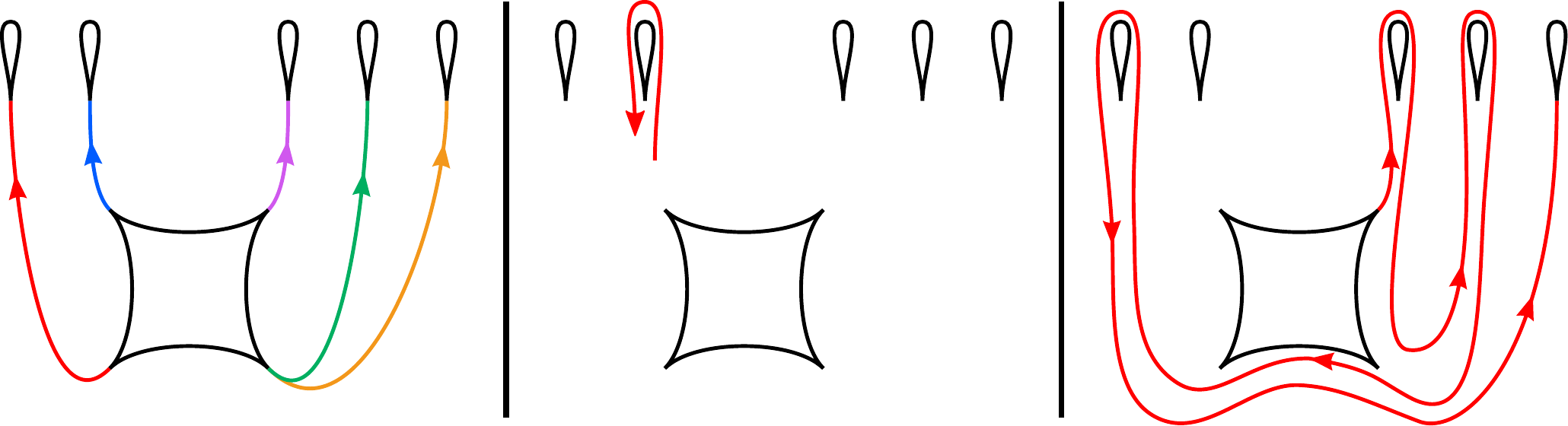}
    \caption{Left: the jellyfish track, with labeled edges. Center: $b^+$ appears in $f_\beta(r)$. Right: $f_\beta(r)=p^-g^-r^+y^\circ$.}
    \label{fig:jellyfishtrackex}
\end{figure}
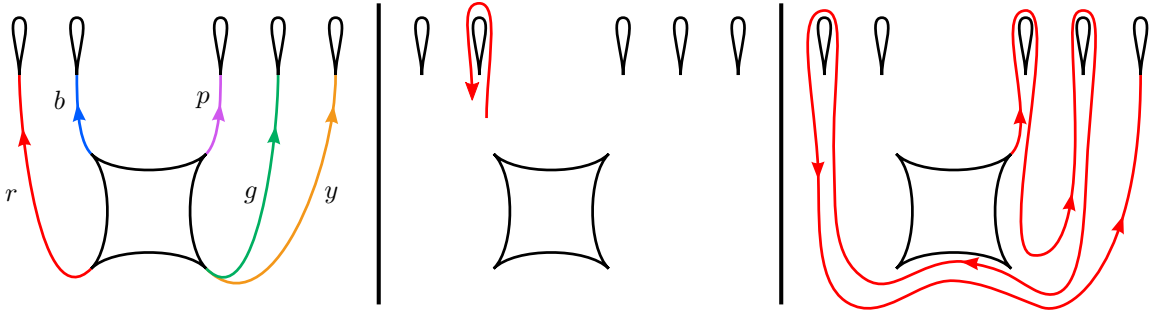

For the rest of this section, $\tau$ will denote the Jellyfish track shown in \ref{fig:jellyfishtrackex}. The real edges are labeled $r$, $b$, $p$, $g$, and $y$ as in the figure, and each is oriented toward the marked points. The data $(\beta,\psi_\beta,\tau,f_\beta)$ will describe a pseudo-Anosov braid carried by $\tau$, which lifts to a map $(h,\psi_h,\widetilde{\tau},f_h)$ on $S$. Note that every real edge $e$ of $\tau$ ends at a jointless 1-marked monogon. In particular, Corollary \ref{lem:tracemon} applies, so if $\psi_h$ is FPF then $e$ never appears in $f_\beta(e)$.

For any real edge $e$, we may write $$f_\beta(e)=e(1)\cdot\overline{e(1)}\cdot e(2)\cdot\overline{e(2)}\cdot~...~\cdot e(l-1)\cdot\overline{e(l-1)}\cdot e(l)$$ where each $e(i)$ is a real edge of $\tau$, and $\overline{e(i)}$ denotes $e(i)$ with orientation reversed. To simplify and specify our notation to the relevant analysis, we will instead write $$f_\beta(e)=e(1)^{\pm}e(2)^{\pm}...e(l-1)^{\pm}e(l)^\circ$$ Here, the $\pm$ in $e(i)^\pm$ describes which way the image $\psi_\beta(e)$ traverses $e(i)$, as in the center of Figure \ref{fig:jellyfishtrackex}, and $e(l)^\circ$ simply means that $f_\beta(e)$ ends at $e(l)$. See the right of Figure \ref{fig:jellyfishtrackex} for an example of this notation. We will also write, e.g., $f_\beta(b)=...r^{+\circ}...$ to indicate that, at that particular instance of $r$ in $f_\beta(b)$, we aren't sure whether $f_\beta(b)$ passes $r$ on the right or stops there.

To begin our analysis, note that the interior infinitesimal quadrilateral of $\tau$ must rotate: if it doesn't, then, for example $f_\beta(r)$ would start at $r$, immediately violating the trace lemma. This rotation is determined by the first letter in $f_\beta(r)$. First, suppose that $f_\beta(r)$ starts at $b$.

\begin{figure}
    \centering
    \includegraphics[width=\textwidth]{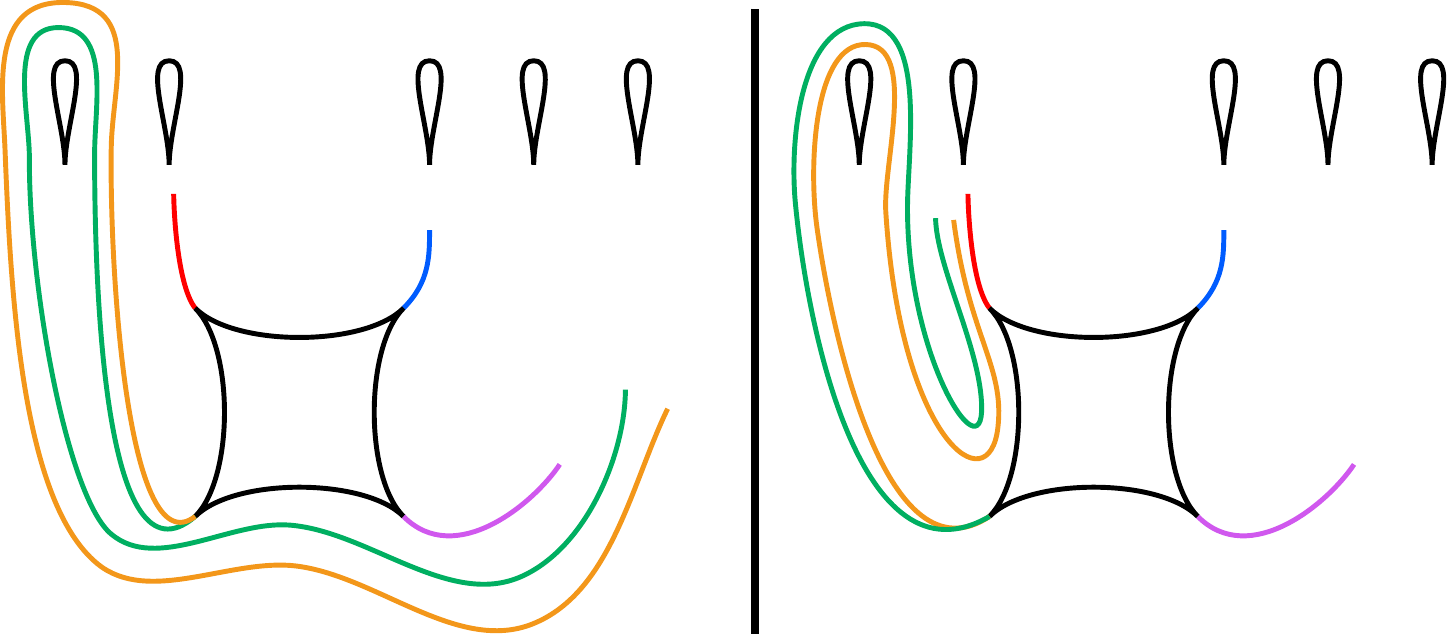}
    \caption{Visual aids for the proof of Lemma \ref{lem:4trackgr+yr-}.}
    \label{fig:4trackgr+yr-}
\end{figure}

\begin{lem}\label{lem:4trackgr+yr-}
If $\psi_h$ is FPF and $f_\beta(r)$ starts at $b$, then $f_\beta(g)$ does not start with $r^+$ and $f_\beta(y)$ does not start with $r^-$.
\end{lem}
\begin{proof}
First, if $f_\beta(g)$ starts with $r^+$, then so does $f_\beta(y)$. The next letter in each of $f_\beta(g)$ and $f_\beta(y)$ must be either $g$ or $y$, as in the left of Figure \ref{fig:4trackgr+yr-}. We know that $f_\beta(g)\neq r^+g^{\pm\circ}...$ by trace, so we must have $f_\beta(g)=r^+y^{\pm\circ}...$, but then $f_\beta(y)=r^+y^{\pm\circ}...$

Next, if $f_\beta(y)$ starts with $r^-$, then so does $f_\beta(g)$. The next letter in both $f_\beta(g)$ and $f_\beta(y)$ is $b$, as in the right of Figure \ref{fig:4trackgr+yr-}. We can see that if $f_\beta(g)=r^-b^{+\circ}...$, or if $f_\beta(g)=r^-b^-r^{\pm\circ}...$ then $f_\beta(r)=b^+r^{\pm\circ}...$. A similar argument holds for $f_\beta(y)$, so we must have $f_\beta(g)=r^-b^-p^{\pm\circ}...$ and $f_\beta(y)=r^-b^-p^{\pm\circ}...$ The same type of argument holds for all of the following letters in $f_\beta(g)$ and $f_\beta(y)$, so it must be that both $f_\beta(g)$ and $f_\beta(y)$ start with $r^-b^-p^-$ and then traverse either $g$ or $y$ next. From here, it's easy to see that either $f_\beta(g)$ passes over $g$ or $f_\beta(y)$ passes over $y$.
\end{proof}

\begin{figure}
    \centering
    \includegraphics{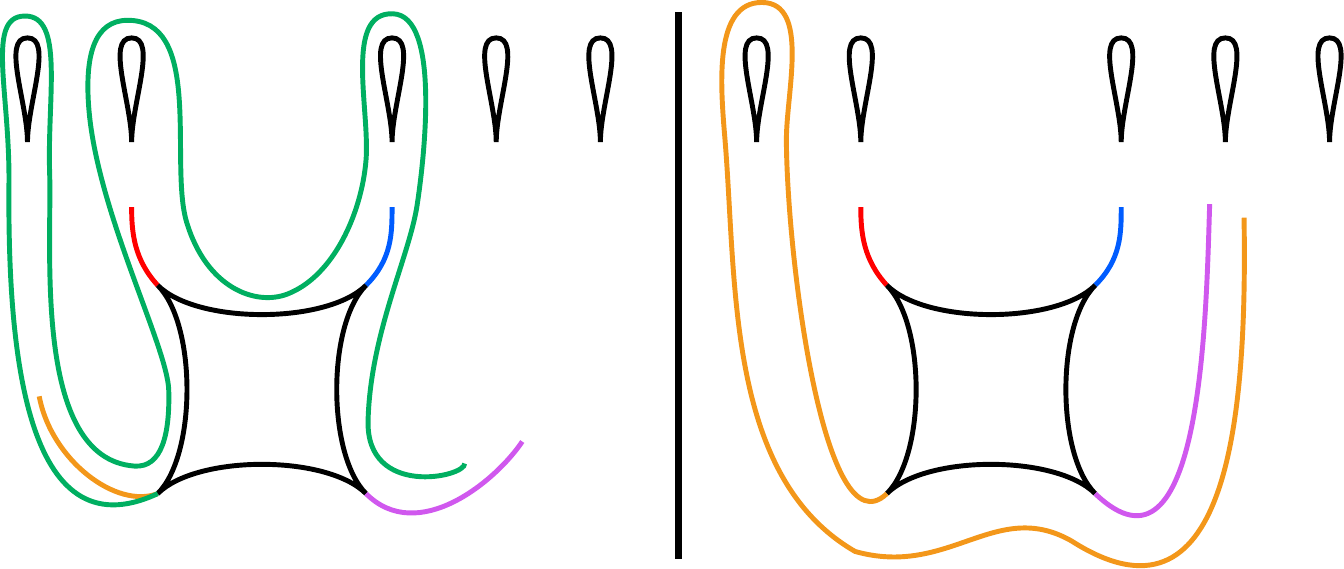}
    \caption{Visual aids for the proof of Lemma \ref{lem:4trackgr-yr+}.}
    \label{fig:4trackgr-yr+}
\end{figure}

\begin{lem}\label{lem:4trackgr-yr+}
If $\psi_h$ is FPF and $f_\beta(r)$ starts at $b$, then $f_\beta(g)$ does not start with $r^-$ and $f_\beta(y)$ does not start with $r^+$.
\end{lem}
\begin{proof}
If $f_\beta(g)$ starts with $r^-$, then a very similar argument to that of the previous lemma shows that $f_\beta(g)=r^-b^-p^-y^{\pm\circ}...$ Note that we must have $f_\beta(g)=r^-b^-p^-y^-r^{\pm\circ}...$, because otherwise $f_\beta(p)$ passes over $p$. So, we must have $f_\beta(g)=r^-b^-p^-y^-r^-b^-p^-...$, as depicted on the left of Figure \ref{fig:4trackgr-yr+}. But, now, $f_\beta(g)$ either eventually passes over $g$, or continues to spiral around the track.

If $f_\beta(y)$ starts with $r^+$, then we must have $f_\beta(y)=r^+g^{\pm\circ}...$, as shown on the right of Figure \ref{fig:4trackgr-yr+}. If $f_\beta(y)=r^+g^{-\circ}...$ then $f_\beta(p)=g^-r^-b^{\pm\circ}...$ The argument here is now similar to the previous ones: if $f_\beta(p)=g^-r^-b^{+\circ}...$ or passes $r$ next, then $f_\beta(r)$ passes over $r$. So, we must have $f_\beta(p)=g^-r-b^-p^{\pm\circ}...$

It follows that $f_\beta(y)=r^+g^+p^{\pm\circ}$. A very similar argument applies for the next few letters, so $f_\beta(y)=r^+g^+p^+b^+r^+...$ and either eventually stops at $y$ or continues to spiral around the outside of the track.
\end{proof}

We can conclude from the lemmas above that $f_\beta(r)$ cannot start at $b$: if it did, then we would have $f_\beta(g)=f_\beta(y)=r^\circ$, but both $f_\beta(g)$ and $f_\beta(y)$ cannot end at $r$. The next cases for where $f_\beta(r)$ starts are almost identical. For example, if $f_\beta(r)$ starts at $p$, then $f_\beta(g)$ and $f_\beta(y)$ both start at $b$. One can then apply the analogous arguments as above to show that in this case $f_\beta(g)=f_\beta(y)=b^\circ$.

\section{The stratum $(2;\emptyset;3^4)$}\label{sec:23333}


\begin{figure}[htp]
    \labellist
    \large
    \pinlabel $\sigma_1^{-1}\sigma_2^{-1}\sigma_3^{-1}$ at 303 720
    \pinlabel $\sigma_4$ at 212 660
    \pinlabel $\sigma_1^{-1}$ at 180 572
    \pinlabel $\sigma_4\sigma_3\sigma_2$ at 432 575
    \pinlabel $\sigma_1^{-1}\sigma_2^{-1}$ at 180 467
    \pinlabel $\sigma_4\sigma_3$ at 400 317
    \pinlabel $\sigma_4$ at 175 235
    \pinlabel $\sigma_1^{-1}\sigma_2^{-1}\sigma_3^{-1}$ at 220 125
    \pinlabel $\sigma_1^{-1}$ at 375 123
    \pinlabel $\sigma_4\sigma_3\sigma_2$ at 320 20
    \endlabellist
    \centering
    \includegraphics[width=0.95\textwidth]{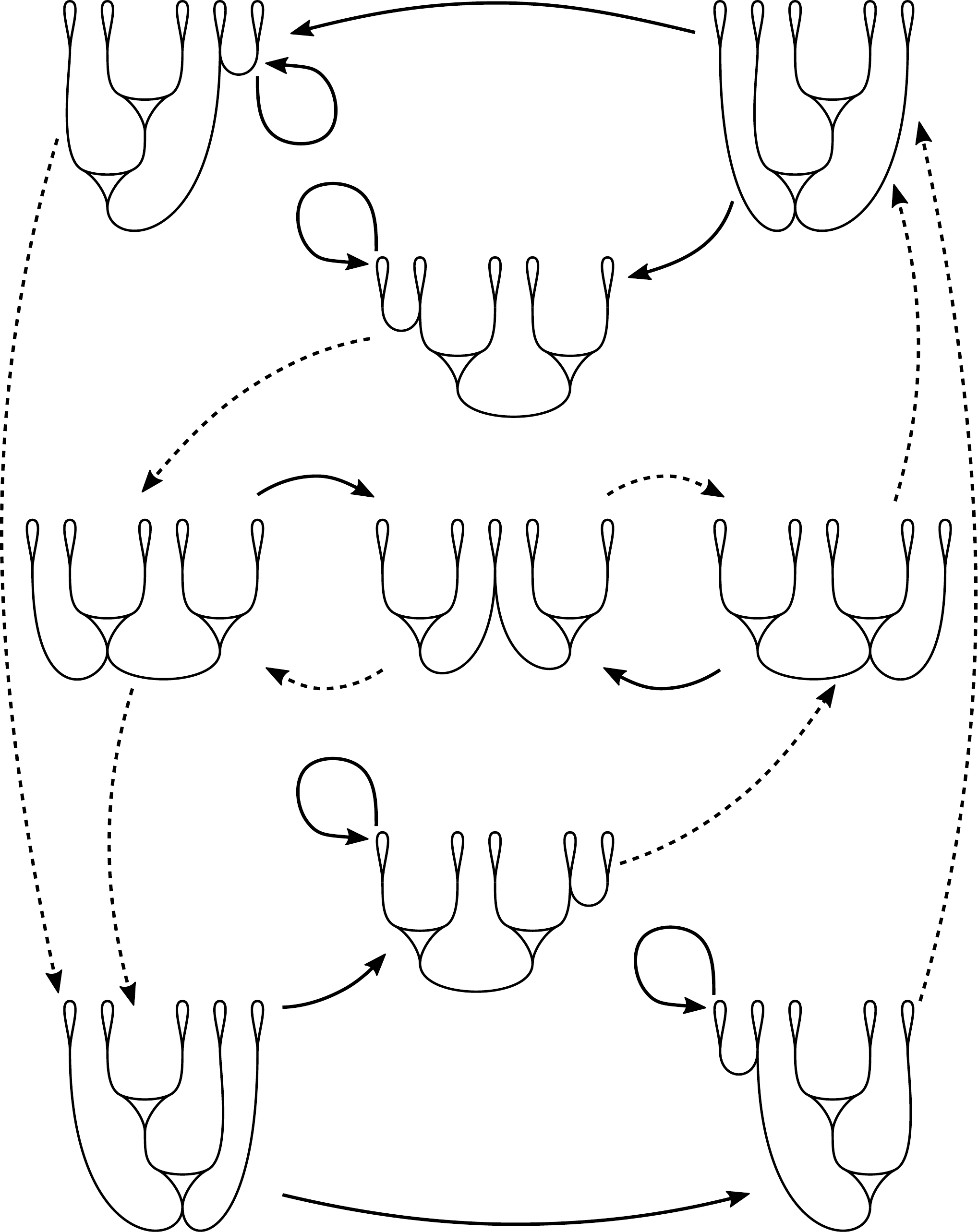}
    \caption{The folding automaton for the stratum $(1;1^5;3^2)$ on $D_5$. The dashed edges induce trivial braid words. Any circuit around the outside of the automaton yields a reducible braid.}
    \label{fig:automaton}
\end{figure}
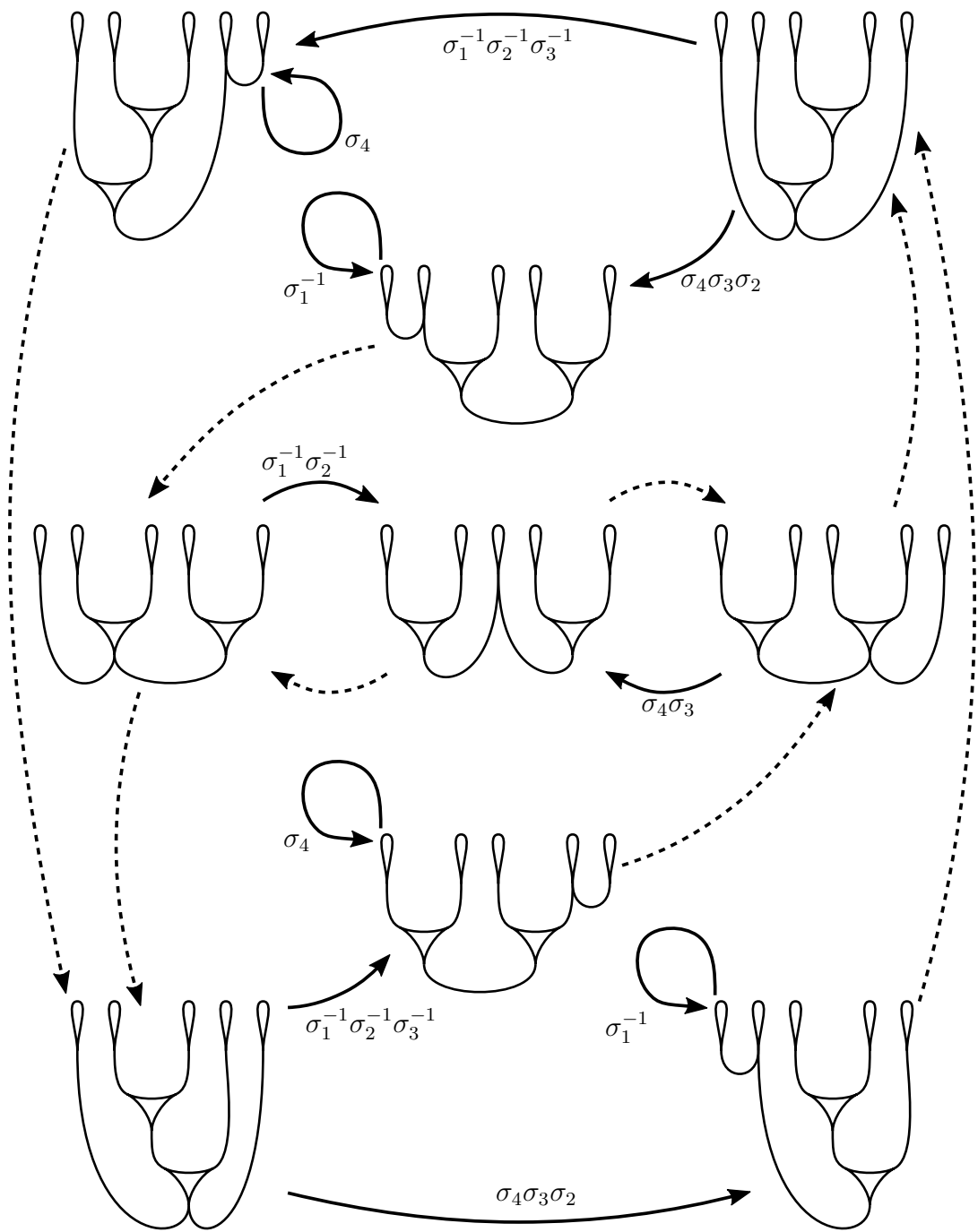

There is one more stratum to consider for the proof of Theorem \ref{thm:FPF}. This last stratum is quite a bit more complicated than the previous ones. Here is what we will prove:

\begin{subthm}\label{thm:2-34}
Let $h\in\text{Mod}(S)$ be a pseudo-Anosov mapping class with geometric representative $\psi_h$. Suppose that $\psi_h$ is FPF and has singularity type $(2;\emptyset;3^4)$. Then, $(S,h)$ is not an open book decomposition for $\cal{P}$.
\end{subthm}

\begin{figure}[t]
    \labellist
    \large
    \pinlabel $\text{The Enoki}$ at 183 240
    \pinlabel $\text{The Camel}$ at 183 80
    \endlabellist
    \centering
    \includegraphics[width=0.8\textwidth]{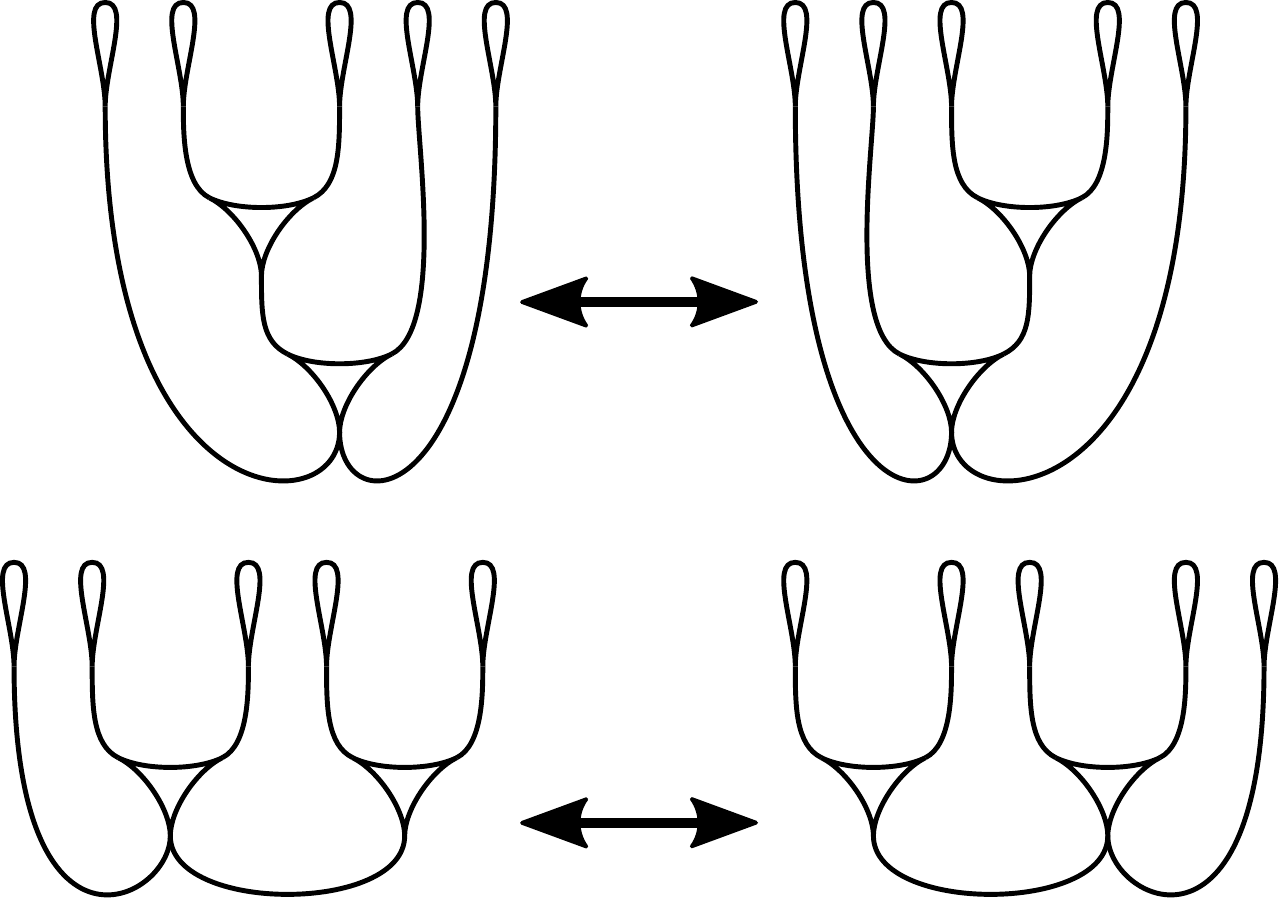}
    \caption{The two pairs of jointless tracks in the stratum $(1;1^5;3^2)$. Each pair is preserved under a horizontal reflection of the plane. Every pseudo-Anosov map is carried by one of the Camel tracks.}
    \label{fig:33trackpairs}
\end{figure}

Together with Theorems \ref{thm:433}, \ref{thm:6}, and \ref{thm:244}, Theorem \ref{thm:2-34} will complete the proof of Theorem \ref{thm:FPF} as explained in the outline (subsection \ref{sec:outline}), which implies Theorem \ref{thm:lspace} as explained in subsection \ref{sec:floer}. So, the rest of the paper will comprise the proof of Theorem \ref{thm:2-34}.

Any $h$ from Theorem \ref{thm:2-34} is the lift of a pseudo-Anosov braid in the stratum $(1;1^5;3^2)$. We will show that there is a canonical track that carries all pseudo-Anosov maps in this stratum, and then check train track maps on that track.

\begin{prop}\label{prop:33track}
Any pseudo-Anosov map (or its reverse) in the stratum $(1;1^5;3^2)$ is conjugate to one carried by the Camel track on the right in Figure \ref{fig:33trackpairs}.
\end{prop}
\begin{proof}
By Theorem \ref{thm:jointless}, any pseudo-Anosov in this stratum is carried by a jointless track. The only jointless tracks in this stratum are those shown in Figure \ref{fig:33trackpairs}: the Enoki pair and the Camel pair. Note that the tracks shown in pairs are related by a horizontal reflection on the disk. In particular, if a map is carried by one track from a pair, then its reverse is carried by the partner track. So, it suffices to show that any pseudo-Anosov map is carried by one of the two Camel tracks.

Ham and Song in \cite{HS} compute the folding automaton for train tracks in the stratum $(1;1^5;3^2)$. This automaton is depicted in Figure \ref{fig:automaton}. The key fact for our proof is that any pseudo-Anosov is represented by a loop in the automaton, and a map $\psi_\beta$ is carried by a track $\tau$ if and only if $\tau$ appears in the automaton loop representing $\psi_\beta$. From this perspective, changing the starting track of such a loop amounts to conjugating $\beta$.

Now, let $\beta(a,b)= \sigma_4^a\sigma_3\sigma_2\sigma_1^{-b}\sigma_2^{-1}\sigma_3^{-1}$. The braid $\beta(a,b)$ is given by an ``outside" loop in the folding automaton, which only passes through the tracks in the corners. Note that any loop starting at one of the Enoki tracks which doesn't pass through either Camel track is given by a product of $\beta(a,b)$'s. But, $\beta(a,b)$ is reducible: it is conjugate to $(\sigma_3^{-1}\sigma_4^a\sigma_3)(\sigma_2\sigma_1^{-b}\sigma_2^{-1})$, which fixes the isotopy class of a curve surrounding only the second and fourth marked points. Similarly, any braid of the form $\beta(a_1,b_1)\beta(a_2,b_2)...\beta(a_n,b_n)$  is reducible, where $a_i,b_i,n\geq 0$ for all $i$. It follows that any pseudo-Anosov in this stratum is carried by one of the Camel tracks, up to conjugation.
\end{proof}

\begin{figure}
    \labellist
    \pinlabel $r$ at 0 45
    \pinlabel $b$ at 52 45
    \pinlabel $p$ at 92 45
    \pinlabel $g$ at 147 45
    \pinlabel $y$ at 180 45
    \pinlabel $d$ at 60 15
    \endlabellist
    \centering
    \includegraphics{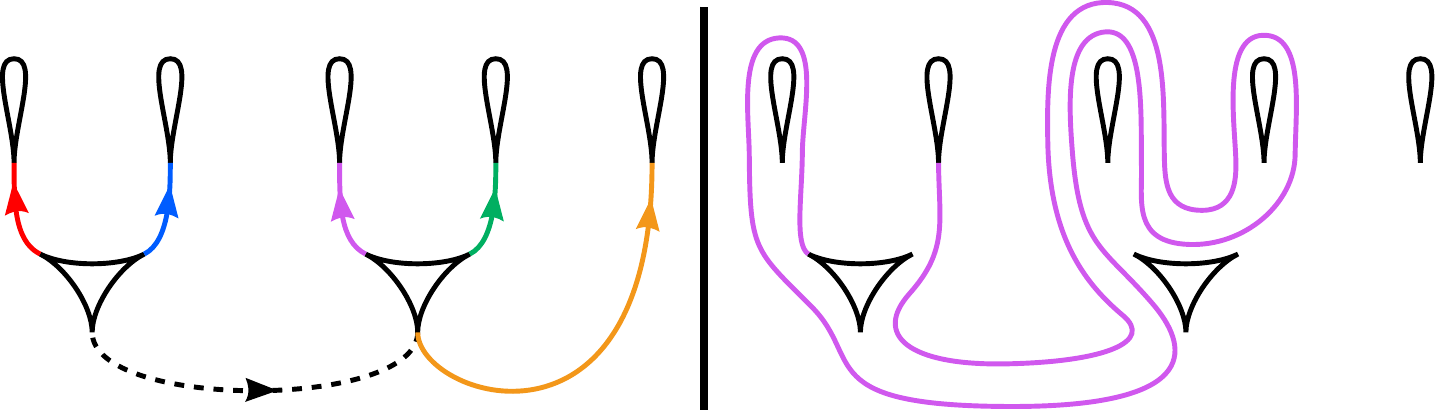}
    \caption{Left: the Camel track. Right: $f_\beta(p)=r^+dp^-g^+p^+\overline{d}b^\circ$}
    \label{fig:cameltrackex}
\end{figure}
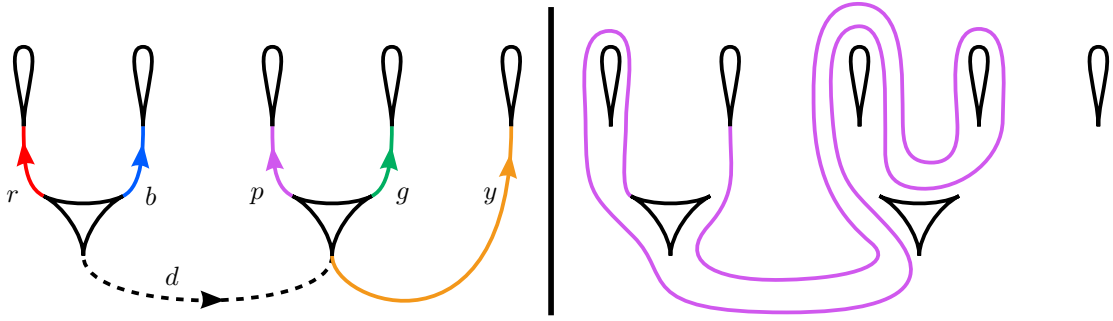

\subsection{Analyzing the candidate braids}
This singularity type does lead to some FPF maps $\psi_h$, but we will show that none of the corresponding mapping classes $h$ describe open book decompositions of the Poincar\'e sphere $\cal{P}$. First, we introduce the candidates.

\begin{figure}[t]
    \centering
    \includegraphics[width=\textwidth]{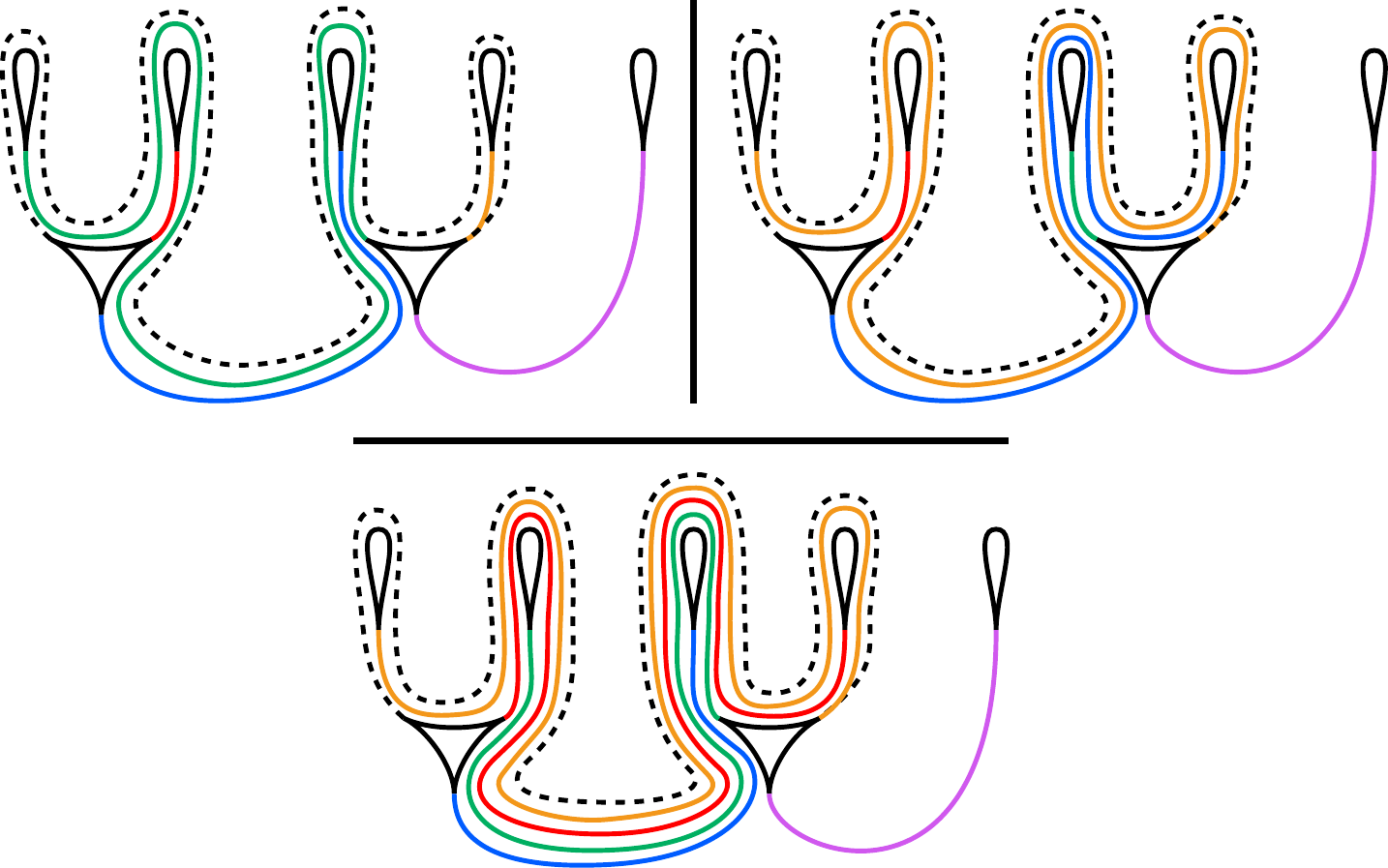}
    \caption{The images $\beta_i(\tau)$ before collapsing. Left: $\beta_1$. Right: $\beta_2$. Bottom: $\beta_3$.}
    \label{fig:beta_i}
\end{figure}

\begin{prop}\label{prop:candidates}
The braids
\begin{align*}
\beta_1&=(\sigma_4\sigma_3)^2(\sigma_2\sigma_1)^{-2}\\
\beta_2&=\sigma_1^{-3}\sigma_2^{-1}\sigma_3^{-1}\sigma_2(\sigma_3\sigma_4)^2\\
\beta_3&=(\sigma_4\sigma_3\sigma_1^{-1}\sigma_2^{-1})^2
\end{align*}
are all pseudo-Anosov and carried by $\tau$. The images $\beta_i(\tau)$ are shown in Figure \ref{fig:beta_i} immediately before collapsing onto $\tau$.
The braids $\Delta^{4k+2}\beta_i$ lift to FPF maps on $S$, for any $i\in\{1,2,3\}$ and any $k\in\bb{Z}$.
\end{prop}
\begin{proof}
A routine isotopy rel. marked points certifies that the images $\beta_i(\tau)$ depicted in Figure \ref{fig:beta_i} are correct. From there, it is immediate that all three $\beta_i$ are carried by $\tau$: the images are transverse to the leaves of an appropriate fibered neighborhood of $\tau$. To check that the braids are pseudo-Anosov, one may compute their transition matrices $M(f_{\beta_i})$ from the images depicted. The matrices are all Perron--Frobenius, so the braids are pseudo-Anosov by Propostion \ref{prop:fix}. The fact that the braids $\Delta^2\beta_i$ lift to FPF maps can be verified by e.g. XTrain \cite{xtrain}.
\end{proof}
\begin{rem}
The braids $\beta_i$ themselves do not lift to FPF maps, since their lifts fix all four 3-pronged singularities. Composition with $\Delta^2$ lifts to composition with the hyperelliptic involution $\iota:S\to S$, which swaps the 3-pronged singularities in pairs.
\end{rem}

\begin{prop}\label{prop:classify}
If a pseudo-Anosov braid $\beta$ is carried by $\tau$ and lifts to a FPF map $h$ in the cover, then $\beta$ is conjugate to one of the $\beta_i$, up to powers of the full twist $\Delta^2$.
\end{prop}

Assuming Proposition \ref{prop:classify}, we will prove Theorem \ref{thm:2-34}:
\begin{proof}[proof of Theorem \ref{thm:2-34}]
With the assumptions of the theorem statement, we can conclude that $h$ is symmetric and projects to a pseudo-Anosov 5-braid $\beta$. Because $\cal{P}$ is the double cover of $S^3$ branched along the closure $\widehat{\beta}$, we know that $\widehat{\beta}=T(3,5)$. And, since $\cal{P}$ is an L-space, we know that $|c(\beta)|<2$ by \cite{HM}. By Proposition \ref{prop:33track}, we know $\beta$ (or its reverse) is conjugate to a map carried by the Camel track from Figure \ref{fig:33trackpairs}. By Proposition \ref{prop:classify}, we know $\beta$ is conjugate to $\Delta^{2k}\beta_i^{\pm1}$ for some $k\in\bb{Z}$. We can easily compute $c(\Delta^{2k}\beta_i^{\pm1})=k$, so we know $\beta$ is conjugate to $\beta_i^{\pm1}$ or $\Delta^2\beta_i^{\pm1}$ for some $i$. It thus suffices to check that none of those braids close up to $T(3,5)$. For each, a determinant computation will work (their determinants are all either 5 or 9, but $\det T(3,5)=1$).
\end{proof}

The rest of this section will be devoted to proving Proposition \ref{prop:classify}. We will do that by an argument similar to that of Section \ref{sec:244}, by carefully analyzing the possible train track maps on the canonical track $\tau$.

\subsection{Setup for the case analysis}

For the rest of this section, $\tau$ will denote the Camel track shown in Figure \ref{fig:cameltrackex}. Denote the real edges of $\tau$ by $r$, $b$, $g$, $p$, $y$ all oriented upward, and $d$ (for dashed) oriented to the right. As before, we will simplify notation to write $e\overline{e}$ as just $e$ for any $e\neq d$; the edge $d$ never appears in a train path twice in a row, and the orientation of $d$ or $\overline{d}$ will be an important consideration later. The decorations $+$, $-$, and $\circ$ over a real edge $e\neq d$ will notate passing $e$ on the right or left or ending there, respectively, as in Section \ref{sec:244}; for $d$, we will only use $\circ$ to denote ending on $d$. An example of this notation is shown in Figure \ref{fig:cameltrackex}.

Next, we will interpret the Trace Lemma (Lemma \ref{lem:trace}) on $\tau$, and see how it interacts with the singularity type of the lifted map. Note that every real edge besides $d$ ends at a 1-marked monogon, so if $\beta$ lifts to a FPF map under the Birman--Hilden correspondence, then $e$ does not appear in $f_\beta(e)$ for $e\neq d$. For $d$, we know there are an even number of letters between any two occurrences of $d$ or $\overline{d}$ in $f_\beta(d)$.

To further restrict the behavior of $f_\beta(d)$, we may examine the interaction of the infinitesimal triangles with the lift $\psi_h$ of $\psi_\beta$ to the surface $S$. Each infinitesimal triangle contains a 3-pronged singularity, which lifts to two 3-pronged singularities in the lift. Note that $h$ commutes with a hyperelliptic involution $\iota:S\to S$, by construction. Denote the four 3-pronged singularities in the lift by $1,2,3,4$, and suppose $\iota(1)=3$, and $\iota(2)=4$. Because $\psi_h$ permutes the set $\{1,2,3,4\}$; $\psi_h$ and $\iota$ commute; and $\psi_h$ is FPF, there are only a few possibilities for $\psi_h(1),...,\psi_h(4)$. The following table lists all the possibilities:

    \begin{center}
    \begin{tabular}{c|c|c|c|c}
     & Case A & Case B & Case C & Case D \\
     \hline
    $\psi_h(1)$ & 3 & 2 & 4 & 2\\
    $\psi_h(2)$ & 4 & 1 & 1 & 3\\
    $\psi_h(3)$ & 1 & 4 & 2 & 4\\
    $\psi_h(4)$ & 2 & 3 & 3 & 1\\
    \end{tabular}
    \end{center}

The upshot of this analysis is that the infinitesimal triangles in $\tau$ downstairs are fixed by $f_\beta$ if and only if their lifts are sent by $f_h$ to the other side of $\widetilde{\tau}$ (case A). Conversely, the infinitesimal triangles in $\tau$ downstairs are swapped by $f_\beta$ if and only if at least one of their lifts is sent to the same side of $\widetilde{\tau}$ (case B: both are sent to the same side, cases C and D: one is sent to the same side, one to the opposite side).

In particular, we can conclude by the Trace Lemma (\ref{lem:trace}) that in case A, when the triangles are fixed by $f_\beta$, that every occurrence of $d$ or $\overline{d}$ in $f_\beta(d)$ occurs after an even number of preceding letters. And, in cases B, C, and D, when the triangles are swapped by $f_\beta$, at least one of $f_\beta(d)$ or $f_\beta(\overline{d})$ passes over $d$ or $\overline{d}$ after odd numbers of letters.

\begin{figure}[h]
    \centering
    \includegraphics[width=\textwidth]{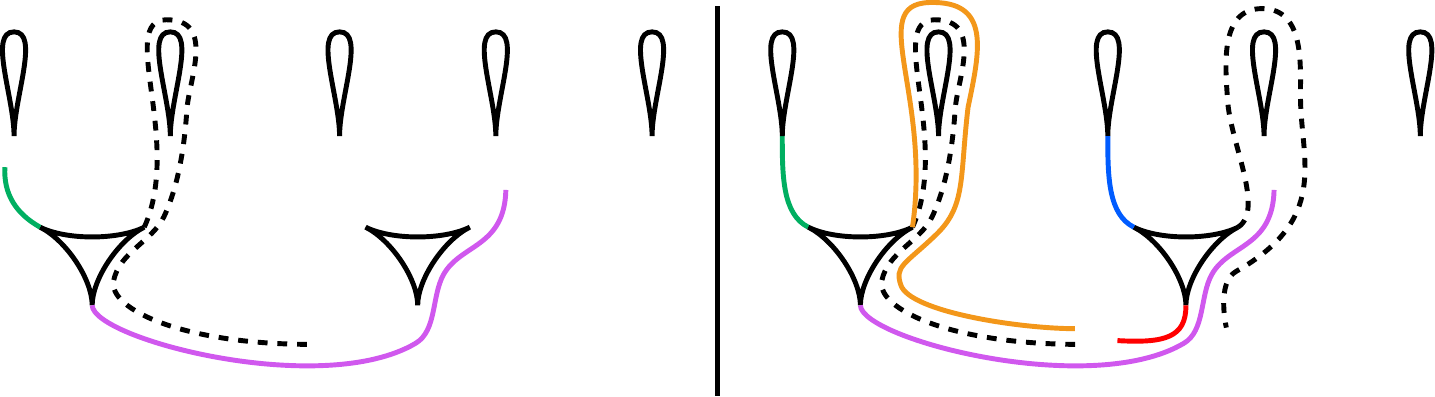}
    \caption{Visual aids for the proof of Lemma \ref{lem:33trackBCDb}.}
    \label{fig:33trackBCDb}
\end{figure}

\begin{figure}[h]
    \centering
    \includegraphics[width=\textwidth]{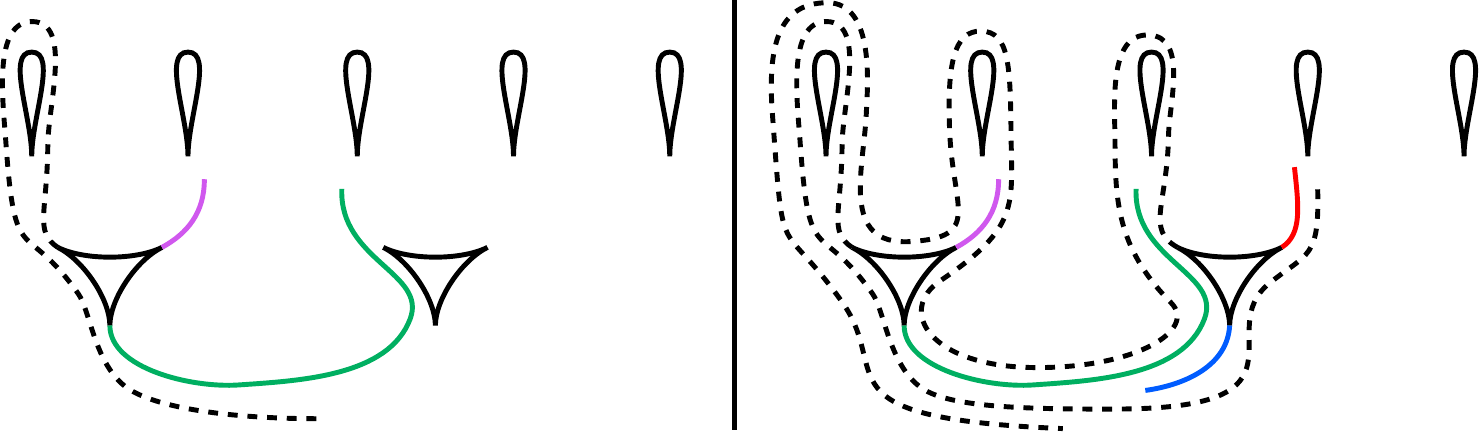}
    \caption{Visual aids for the proof of Lemma \ref{lem:33trackBCDr}.}
    \label{fig:33trackBCDr}
\end{figure}

\begin{figure}[h]
    \centering
    \includegraphics[width=\textwidth]{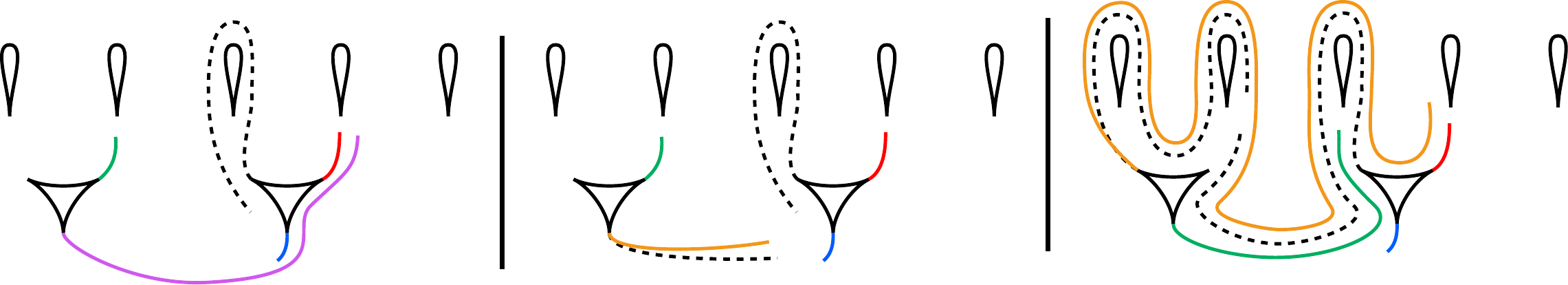}
    \caption{Visual aids for the proof of Lemma \ref{lem:33trackBCDp}.}
    \label{fig:33trackBCDp}
\end{figure}

\begin{figure}[h]
    \centering
    \includegraphics[width=\textwidth]{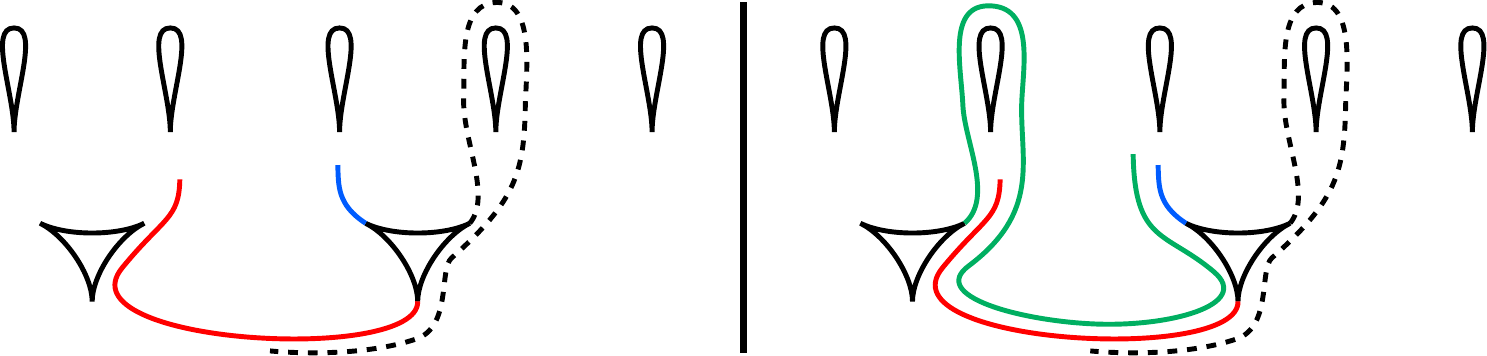}
    \caption{Visual aids for the proof of Lemma \ref{lem:33trackBCDg}.}
    \label{fig:33trackBCDg}
\end{figure}

\begin{figure}[h]
    \centering
    \includegraphics[width=\textwidth]{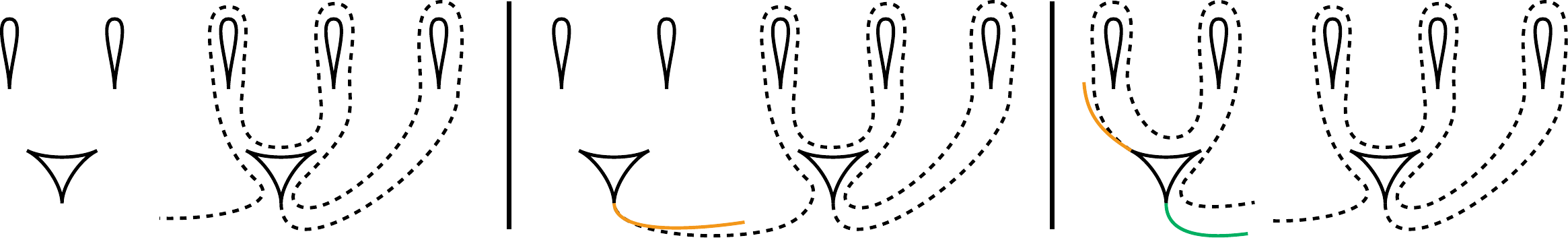}
    \caption{Visual aids for the proof of Lemma \ref{lem:33trackBCDy}.}
    \label{fig:33trackBCDy}
\end{figure}

\FloatBarrier

\subsection{Cases B, C, D}
Our analysis begins with cases B,C, and D above, when the triangles are swapped by $f_\beta$. In these cases, we know that at least one of $f_\beta(d)$ or $f_\beta(\overline{d})$ passes over $d$ or $\overline{d}$ after an odd number of letters. We will build our case analysis around which of $f_\beta(d)$ or $f_\beta(\overline{d})$ satisfies this property. In what follows, we \textit{strongly} encourage the reader to grab some colored writing implements and draw the maps described in each proof: the symbols can only help so much.

First, assume that $f_\beta(\overline{d})$ passes over $d$ after an odd number of letters. We know that $f_\beta(\overline{d})$ starts at $b$, $r$, or $d$ because the triangles are swapped by $f_\beta$. But, $f_\beta(\overline{d})$ cannot start at $d$, because by assumption it passes over $d$ after an odd number of letters. So, it must be that $f_\beta(\overline{d})$ starts at either $b$ or $r$.

\begin{lem}\label{lem:33trackBCDb}
If $f_\beta(\overline{d})$ starts at $b$, then $\psi_h$ has a fixed point.
\end{lem}
\begin{proof}
Because $f_\beta(\overline{d})$ starts at $b$, we must have that $f_\beta(g)$ starts at $r$ and $f_\beta(p)$ starts at $d$. Now, note that $f_\beta(p)=dg^{\pm\circ}...$ by trace. It follows that either $f_\beta(g)=r^\circ$, or $f_\beta(g)=r^+b^{\pm\circ}...$ by trace. Note that $f_\beta(\overline{d})=b^-d...$ (possibly after some initial twisting over $r$ and $b$), and then it's easy to see that $f_\beta(g)=r^\circ$ or $f_\beta(g)=r^-b^-dp^{\pm\circ}...$ We're now in the situation pictured on the left of Figure \ref{fig:33trackBCDb} (with some extra information about $f_\beta(g)$ not yet shown).

From here, look at the other triangle. We don't know where each edge starts on the right triangle, so consider which edge starts at $g$. If $f_\beta(r)$ starts at $g$, then either $f_\beta(r)$ will pass over $r$ or $f_\beta(p)$ will pass over $p$. We reach a similar conclusion if $f_\beta(b)$ starts at $g$.

The slightly harder case is if $f_\beta(d)$ starts at $g$. In this case, look at $f_\beta(b)$: we know it starts at $p$. If $f_\beta(b)$ starts with $p^+$ then it will pass over $b$. On the other hand, if $f_\beta(b)$ starts with $p^-$ then either $f_\beta(b)$ will pass over $b$ or $f_\beta(p)$ will pass over $p$. It follows that $f_\beta(b)=p^\circ$. From there, we can quickly conclude that $f_\beta(g)=r^\circ$.

Next, follow $f_\beta(\overline{d})$ and $f_\beta(y)$. We must have $f_\beta(\overline{d})=b^-dp^-g^-...$ and then $f_\beta(\overline{d})$ must spiral around the track some (possibly 0) number of times before eventually passing over $y$. This situation is pictured on the right in Figure \ref{fig:33trackBCDb}. Finally, note that $f_\beta(y)$ follows all of $f_\beta(\overline{d})$ to that point, so $f_\beta(y)$ passes over $y$.
\end{proof}

\begin{lem}\label{lem:33trackBCDr}
If $f_\beta(\overline{d})$ starts at $r$, then $\psi_h$ has a fixed point.
\end{lem}
\begin{proof}
In this case, because $f_\beta(\overline{d})$ passes over $d$ after an odd number of letters, we must have $f_\beta(\overline{d})=r^+d...$ (possibly after some initial twisting on $r$ and $b$). Also, we know $g$ starts at $d$, so we must have $f_\beta(g)=dp^\pmd...$ by trace. It is easy to see from here that either $f_\beta(p)=b^\circ$ or $f_\beta(p)=b^+r+dg^{\pm\circ}...$ This situation is pictured on the left of Figure \ref{fig:33trackBCDr}.

Now, consider which edge starts at $p$; whichever edge it is, it must start with $p^+$, because otherwise $f_\beta(g)$ will be forced to pass over $g$. If $f_\beta(b)$ starts at $p$, we know $f_\beta(b)=p^+\overline{d}b^{\pm\circ}...$ and if $f_\beta(r)$ starts at $p$, we know $f_\beta(r)=p^+\overline{d}b^+r^{\pm\circ}...$ In either case, we've reached a contradiction by trace.

So, we are left to consider if $f_\beta(d)$ can start at $p$. Note that in this case, $f_\beta(d)=p^+\overline{d}b^+r^+d...$ and $f_\beta(\overline{d})=r^+dg^{\pm\circ}...$ See the right of Figure \ref{fig:33trackBCDr}. From here, because there must be an even number of letters between occurrences of $d$ or $\overline{d}$ in both $f_\beta(d)$ and $f_\beta(\overline{d})$, we can see that $f_\beta(d)$ and $f_\beta(\overline{d})$ will continue to spiral around the outside of the track and never meet.
\end{proof}

This completes the casework under the assumption that $f_\beta(\overline{d})$ passes over $d$ after an odd number of letters. Now, suppose instead that $f_\beta(\overline{d})$ passes over $d$ after an even number of letters. From the argument in the previous subsection, we can then conclude that $f_\beta(d)$ passes over $\overline{d}$ after an odd number of letters. In this case, $f_\beta(d)$ starts at one of $p$, $g$, or $y$.

\begin{lem}\label{lem:33trackBCDp}
If $f_\beta(d)$ starts at $p$, then $\psi_h$ has a fixed point.
\end{lem}
\begin{proof}
In this case, we know $f_\beta(d)$ starts with $p^+$, $f_\beta(b)$ starts with $\overline{d}$ or $y^{\pm\circ}$, and $f_\beta(r)$ starts with $g^{\pm\circ}$. We don't know which edges start where on the left triangle, but consider cases for which edge starts at $d$. The options are: $f_\beta(p)$ starts at $d$, $f_\beta(\overline{d})$ starts at $d$, or $f_\beta(g)$ starts at $d$. If $f_\beta(p)$ starts at $d$, then we must have $f_\beta(p)=dg^{\pm\circ}...$ by trace. See the left of Figure \ref{fig:33trackBCDp}. Then, we can see $f_\beta(b)=\overline{d}b^{\pm\circ}...$, so $f_\beta(p)$ can't start at $d$.

If instead $f_\beta(\overline{d})$ starts at $d$, then a similar argument works by looking at $f_\beta(\overline{d})$ and $f_\beta(b)$. See the center of Figure \ref{fig:33trackBCDp}. We know the next letter in $f_\beta(\overline{d})$ is either $p$ or $g$. If $f_\beta(\overline{d})=dg^{\pm\circ}$, then we must have $f_\beta(b)=\overline{d}b^{\pm\circ}...$ On the other hand, if $f_\beta(\overline{d})=dp^{\pm\circ}...$ then after a quick inspection of $f_\beta(y)$ and $f_\beta(\overline{d})$we can conclude that $f_\beta(y)=dp^-g^{\pm\circ}...$ From here, we can check that either $f_\beta(y)$ passes over $y$, $f_\beta(r)$ passes over $r$, or $f_\beta(g)$ passes over $g$.

The last case is if $f_\beta(g)$ starts at $d$. Here, we know $f_\beta(g)=dp^{\pm\circ}...$ and $f_\beta(d)=p^+\overline{d}...$ Because by assumption $f_\beta(\overline{d})$ passes over $d$ or $\overline{d}$ after an even number of times, we must have $f_\beta(\overline{d})=r^-b^-d...$ and we can quickly conclude that $f_\beta(d)$ and $f_\beta(\overline{d})$ must meet here. So, then, $f_\beta(d)=p^+\overline{d}b^+r^\circ$. We can then see that $f_\beta(y)=r^-b^-dp^-g^{\pm\circ}...$, as in the right of Figure \ref{fig:33trackBCDp}. Looking at $f_\beta(r)$ now, we find $f_\beta(r)=g^\circ$. But, then $f_\beta(y)$ will continue spiraling around the outside of the track unless it eventually passes over $y$.
\end{proof}

\begin{lem}\label{lem:33trackBCDg}
If $f_\beta(d)$ starts at $g$, then $\psi_h$ has a fixed point.
\end{lem}
\begin{proof}
Here, we know $f_\beta(d)=g^-\overline{d}...$ (it's possible that $f_\beta(d)$ twists over $g$ and $y$ some number of times before going to $\overline{d}$, but this won't matter for our analysis). It follows that $f_\beta(r)=\overline{d}b^{\pm\circ}...$. See the left of Figure \ref{fig:33trackBCDg}. Now, consider cases for which edge starts at $b$; whichever edge it is will start with $b^-$ because otherwise $f_\beta(r)$ will pass over $r$.

If $f_\beta(\overline{d})$ starts with $b^-$, then $f_\beta(\overline{d})=b^-d...$, so $f_\beta(\overline{d})$ passes over $d$ after an odd number of letters. But we're assuming this is not the case. And, if $f_\beta(p)$ starts with $b^-$, then it will pass over $p$. Finally, if $f_\beta(g)$ starts with $b^-$, then we must have $f_\beta(g)=b^-dp^{\pm\circ}...$, as in the right of Figure \ref{fig:33trackBCDg}. From there, it is not so hard to see that either $f_\beta(b)$ passes over $b$ or $f_\beta(g)$ passes over $g$.
\end{proof}

\begin{lem}\label{lem:33trackBCDy}
If $f_\beta(d)$ starts at $y$, then $\psi_h$ has a fixed point.
\end{lem}
\begin{proof}
If $f_\beta(d)$ starts at $y$, then we must have $f_\beta(d)=y^+g^+p^+d...$ (possibly after some initial twisting which will not matter for our analysis). See the left of Figure \ref{fig:33trackBCDy}. We'll consider cases for which edge starts with $d$; whichever edge it is must then continue to $p$ afterward. In particular, we know $f_\beta(p)$ cannot start with $d$ here.

If $f_\beta(\overline{d})$ starts with $d$, then we can quickly deduce that $f_\beta(d)$ and $f_\beta(\overline{d})$ must meet here, so that $f_\beta(d)=y^+g^+p^+d^\circ$. See the center of Figure \ref{fig:33trackBCDy}. But, we then have $f_\beta(y)=dp^-g^-y^{\pm\circ}...$

If instead $f_\beta(g)$ starts with $d$, then we also know $f_\beta(\overline{d})$ and $f_\beta(y)$ both start with $r$. By the assumption that $f_\beta(\overline{d})$ passes over $d$ after an even number of letters, we know $f_\beta(\overline{d})=r^-b^-d...$, as in the right of Figure \ref{fig:33trackBCDy}. We can quickly conclude that $f_\beta(d)$ and $f_\beta(\overline{d})$ must meet here, so that $f_\beta(d)=y^+g^+p^+\overline{d}b^+r^\circ$. But, then $f_\beta(y)$ will eventually pass over $y$.
\end{proof}

\begin{figure}
    \centering
    \includegraphics{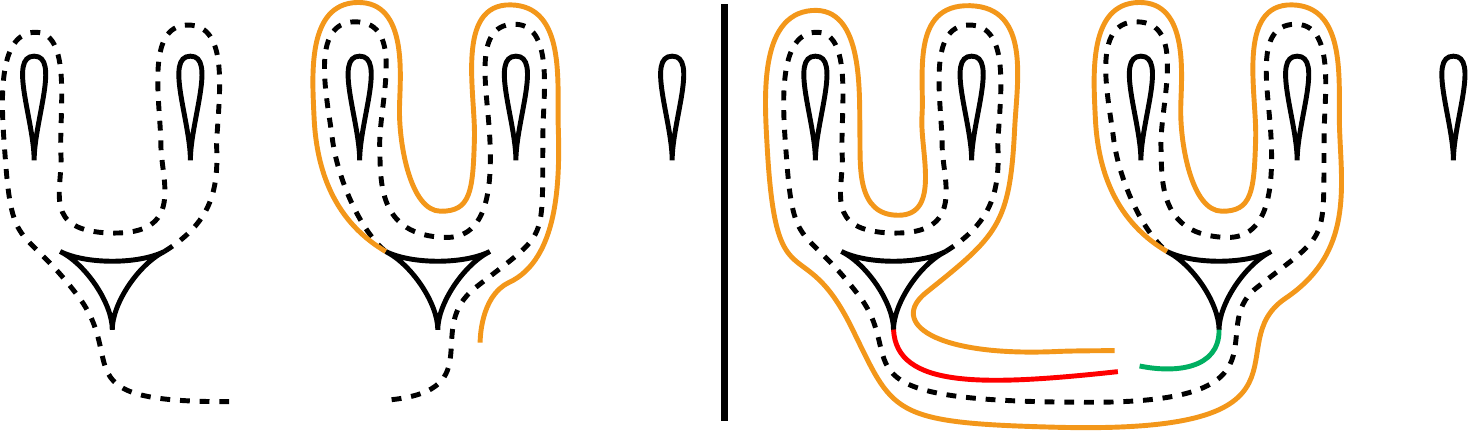}
    \caption{Visual aids for the proof of Lemma \ref{lem:33trackAbp}}
    \label{fig:33trackAbp}
\end{figure}

\begin{figure}
    \centering
    \includegraphics{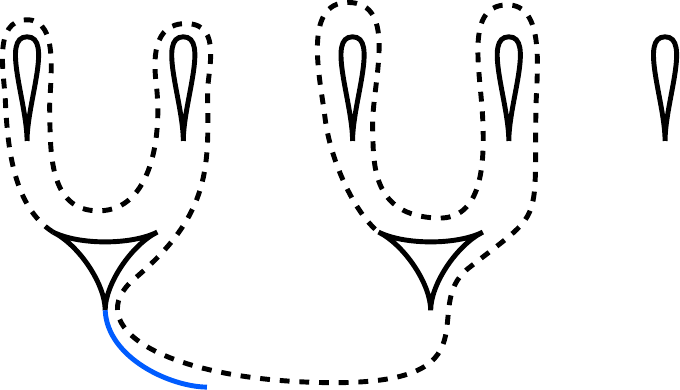}
    \caption{Visual aid for the proof of Lemma \ref{lem:33trackArp}.}
    \label{fig:33trackArp}
\end{figure}

\begin{figure}
    \centering
    \includegraphics{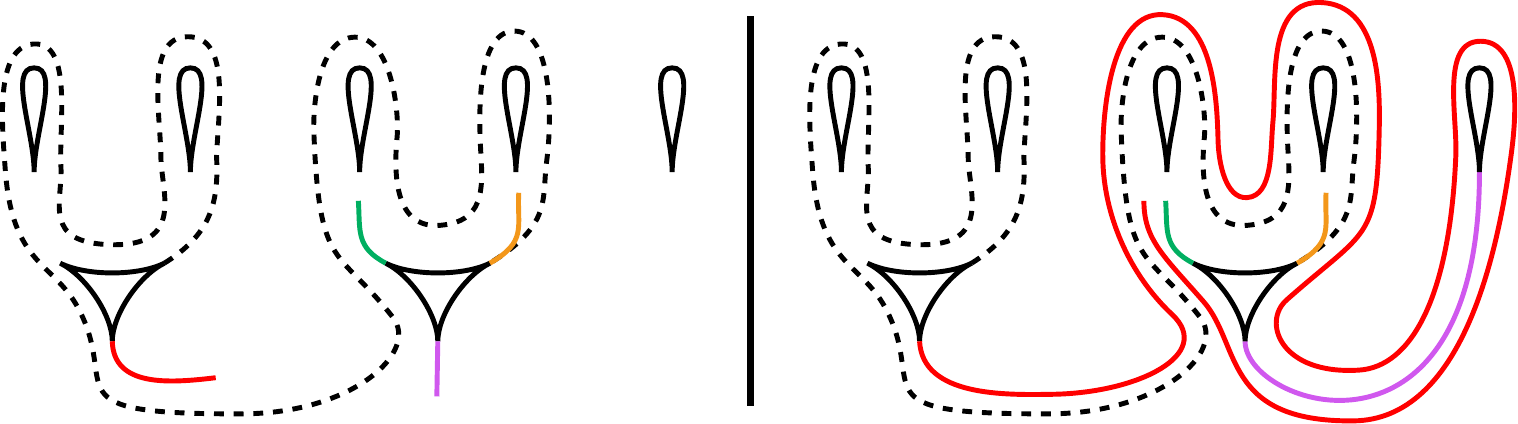}
    \caption{Visual aids for the proof of Lemma \ref{lem:33trackAbg}}
    \label{fig:33trackAbg}
\end{figure}

\FloatBarrier

\begin{figure}[h!]
    \centering
    \includegraphics{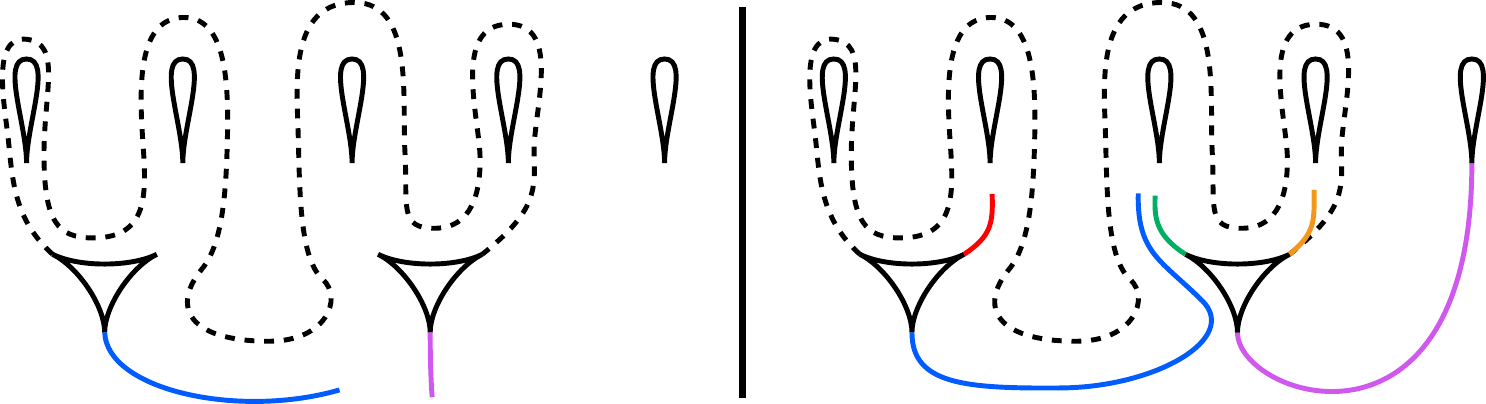}
    \caption{Visual aids for the proof of Lemma \ref{lem:33trackArg}.}
    \label{fig:33trackArg}
\end{figure}

\subsection{Case A}

This is the case in which the candidate braids $\beta_i$ arise, and we will see them appear at the end of this subsection. In case A, recall that we know the infinitesimal triangles of $\tau$ are fixed by $f_\beta$, and that $f_\beta(d)$ and $f_\beta(\overline{d})$ both pass over $d$ after an even number of letters. As in cases B,C,D, we will split our argument into cases based on $f_\beta(d)$ and $f_\beta(\overline{d})$. This time, we know that $f_\beta(d)$ does not start at $d$ and $f_\beta(\overline{d})$ does not start at $\overline{d}$, because otherwise $r$, $b$, $p$, and $g$ would all pass over themselves. We also know that $f_\beta(\overline{d})$ does not start at $y$, because then $y$ would pass over itself.

So, we know $f_\beta(d)$ must start at $r$ or $b$, and $f_\beta(\overline{d})$ must start at $p$ or $g$. The four pairs of choices for $f_\beta(d)$ and $f_\beta(\overline{d})$ are the final cases to complete the proof of Proposition \ref{prop:33track} and Theorem \ref{thm:2-34}.

\begin{lem}\label{lem:33trackAbp}
If $f_\beta(d)$ starts at $b$ and $f_\beta(\overline{d})$ starts at $p$, then $\psi_h$ has a fixed point.
\end{lem}
\begin{proof}
In this case, we can completely work out $f_\beta(d)$. Note that $f_\beta(d)=b^+r^+d...$ and $f_\beta(\overline{d})=p^-g^-...$ because each needs to pass over $d$ after an even number of letters. From here, note that $f_\beta(\overline{d})=p^-g^-\overline{d}...$ because otherwise $y$ passes over itself. This situation is pictured on the left of Figure \ref{fig:33trackAbp}. Now, $f_\beta(d)$ and $f_\beta(\overline{d})$ must meet here: otherwise, whichever end goes ``inside" the other would have to pass over $d$ again after an odd number of letters. For example, if $f_\beta(\overline{d})$ goes above $f_\beta(d)$ and inside it to continue towards $r$, then $f_\beta(\overline{d})$ will eventually be forced to pass over $d$ after an odd number letters.

So, we must have $f_\beta(d)=b^+r^+dg^+p^\circ$. Now, note that $f_\beta(y)=p^-g^-\overline{d}r^-b^-d...$ From here, look at $r$ and $g$. We know that $f_\beta(r)$ starts at $d$ and $f_\beta(g)$ starts at $\overline{d}$. See the right of Figure \ref{fig:33trackAbp}. If $f_\beta(g)$ goes ``above" $r$ here, then we must have $f_\beta(g)$ will eventually pass over $g$. On the other hand, if $f_\beta(r)$ goes ``above" $g$, then either $f_\beta(r)$ will pass over $r$, or $f_\beta(y)$ will pass over $y$ (e.g. if $f_\beta(r)=dp^-g^-y...$)
\end{proof}

\begin{lem}\label{lem:33trackArp}
If $f_\beta(d)$ starts at $r$ and $f_\beta(\overline{d})$ starts at $p$ then $\psi_h$ has a fixed point.
\end{lem}
\begin{proof}
Just as in the previous lemma, we can completely determine $f_\beta(d)$ here. Using a very similar argument to the previous lemma, we can conclude that $f_\beta(d)=r^-b^-dg^+p^\circ$. Now, simply note that $f_\beta(b)=dg^+p^+\overline{d}b^{\pm\circ}...$ See Figure \ref{fig:33trackArp}.
\end{proof}

\begin{lem}\label{lem:33trackAbg}
If $f_\beta(d)$ starts at $b$ and $f_\beta(\overline{d})$ starts at $g$, then $\psi_h$ has a fixed point.
\end{lem}
\begin{proof}
As in the previous two lemmas, we can completely determine $f_\beta(d)$: it must be $f_\beta(d)=b^+r^+dp^-b^\circ$. See the left of Figure \ref{fig:33trackAbg}. Next, look at $f_\beta(p)$. We must have $f_\beta(p)=y^\circ$ by trace. From here, note that $f_\beta(r)=dp^-g^-y^-p^{\pm\circ}...$ as shown on the right of Figure \ref{fig:33trackAbg}. To finish the proof, consider how $f_\beta(r)$ interacts with $f_\beta(g)$ and $f_\beta(y)$.

$f_\beta(r)$ must continue $f_\beta(r)=dp^-g^-y^-p^-b^{\pm\circ}...$ because otherwise $f_\beta(g)$ eventually passes over $g$ (after following $f_\beta(r)$ backwards for a while). But, in that case, $f_\beta(y)$ must eventually pass over $y$ (note that $f_\beta(y)\neq g^\circ$, because then $f_\beta(r)$ is absorbed into the vertex of the triangle near $g$).
\end{proof}

\begin{lem}\label{lem:33trackArg}
If $f_\beta(d)$ starts at $r$ and $f_\beta(\overline{d})$ starts at $g$, then either $\psi_h$ has a fixed point, or $\beta$ is conjugate (up to full twists) to one of the $\beta_i$.
\end{lem}
\begin{proof}
By Theorem \ref{thm:ttMCG}, it suffices to show that $f_\beta=f_{\beta_i}$ for some $i\in\{1,2,3\}$.

In this case, we can conclude that $f_\beta(d)=r^-b^-dp^-g^\circ$. See the left of Figure \ref{fig:33trackArg}. Now, look at $f_\beta(p)$. We can quickly conclude that either $f_\beta(p)=y^\circ$ or $f_\beta(p)=\overline{d}b^{\pm\circ}$. In the latter case, $f_\beta(b)=dg^+p^+\overline{d}b^{\pm\circ}...$ So, we must have $f_\beta(p)=y^\circ$, and then we also have $f_\beta(b)=dp^{\pm\circ}$. See the right of Figure \ref{fig:33trackArg}.

From here, we will look at $f_\beta(b)$, $f_\beta(r)$, $f_\beta(g)$, and $f_\beta(y)$. By considering how the images of these four edges interact, we will find that $f_\beta$ must be one of the train track maps induced by the $\beta_i$. To start, note that either $f_\beta(b)=dp^\circ$ or $f_\beta(b)=dp^-g^{-\circ}...$ (if $f_\beta(b)=dp^+...$ then $f_\beta(b)$ will eventually pass over $b$). In the latter case, we can conclude that $f_\beta(b)=dp^-g^\circ$, because otherwise either $f_\beta(g)$ will pass over $g$ or $f_\beta(b)$ will pass over $b$. So, either $f_\beta(b)=dp^\circ$ or $f_\beta(b)=dp^-g^\circ$.

If $f_\beta(b)=dp^\circ$, then $f_\beta(g)=p^+\overline{d}b^{\pm\circ}...$ By looking at $f_\beta(r)$ and $f_\beta(g)$, we can conclude that either $f_\beta(g)=p^+\overline{d}b^\circ$, or $f_\beta(g)=p^+\overline{d}b^+r^\circ$. In the first case, we must then have $f_\beta(r)=b^-dp^-g^\circ$ and $f_\beta(y)=g^+p^+\overline{d}b^+r^\circ$, which is the train track map induced by $\beta_3$. In the second case, where $f_\beta(g)=p^+\overline{d}b^+r^\circ$, we then have $f_\beta(r)=b^\circ$ and $f_\beta(y)=g^\circ$. This is the train track map for $\beta_1$.

If instead $f_\beta(b)=dp^-g^\circ$, then $f_\beta(g)=p^\circ$ and $f_\beta(y)=g^+p^+\overline{d}b^{\pm\circ}$. Here, we must have $f_\beta(y)=g^+p^+\overline{d}b^+r^\circ$, because otherwise $r$ will follow back along the path of $f_\beta(y)$ and get absorbed into the vertex near $b$. From here, we can then conclude $f_\beta(r)=b^\circ$, and this is the train track map for $\beta_2$.
\end{proof}

\bibliographystyle{alpha}
\bibliography{references}

\begin{thebibliography}{HKM07}

\bibitem[BH95]{BH}
M.~Bestvina and M.~Handel.
\newblock Train-tracks for surface homeomorphisms.
\newblock {\em Topology}, 34(1):109--140, 1995.

\bibitem[BHS21]{BHS}
John~A. Baldwin, Ying Hu, and Steven Sivek.
\newblock Khovanov homology and the cinquefoil, 2021.

\bibitem[Bri00]{xtrain}
Peter Brinkmann.
\newblock {An implementation of the Bestvina-Handel algorithm for surface
  homeomorphisms}.
\newblock {\em Experimental Mathematics}, 9(2):235 -- 240, 2000.

\bibitem[Cau23]{J}
Jacob Caudell.
\newblock personal communication, 2023.

\bibitem[CC09]{CC}
Andrew Cotton-Clay.
\newblock Symplectic floer homology of area-preserving surface diffeomorphisms.
\newblock {\em Geometry \& Topology}, 13(5):2619--2674, 2009.

\bibitem[CH08]{CH}
Jin-Hwan Cho and Ji-Young Ham.
\newblock The minimal dilatation of a genus-two surface.
\newblock {\em Experimental Mathematics}, 17(3):257 -- 267, 2008.

\bibitem[FLP12]{FLP}
Albert Fathi, François Laudenbach, and Valentin Poénaru.
\newblock {\em Thurston's Work on Surfaces (MN-48)}.
\newblock Princeton University Press, Princeton, 2012.

\bibitem[FRW22]{FRW}
Ethan Farber, Braeden Reinoso, and Luya Wang.
\newblock Fixed point-free pseudo-anosovs and the cinquefoil, 2022.

\bibitem[GS22]{GS}
Paolo Ghiggini and Gilberto Spano.
\newblock Knot floer homology of fibred knots and floer homology of surface
  diffeomorphisms, 2022.

\bibitem[HKM07]{hkm}
Ko~Honda, William~H Kazez, and Gordana Mati\'c.
\newblock Right-veering diffeomorphisms of compact surfaces with boundary.
\newblock {\em Invent. Math.}, 169:427--449, 2007.

\bibitem[HM18]{HM}
Matthew Hedden and Thomas~E. Mark.
\newblock Floer homology and fractional dehn twists.
\newblock {\em Advances in Mathematics}, 324:1--39, 2018.

\bibitem[HS07]{HS}
Ji-Young Ham and Won~Taek Song.
\newblock The minimum dilatation of pseudo-ansonov 5-braids.
\newblock {\em Experimental Mathematics}, 16(2):167 -- 180, 2007.

\bibitem[IK17]{ik}
Tetsuya Ito and Keiko Kawamuro.
\newblock Essential open book foliations and fractional dehn twist coefficient.
\newblock {\em Geometriae Dedicata}, 187:17--67, 2017.

\bibitem[KLS02]{KLS}
Ki~Hyoung Ko, J\'er\^ome~E. Los, and Won~Taek Song.
\newblock Entropies of braids.
\newblock {\em Journal of Knot Theory and Its Ramifications}, 11(04):647--666,
  2002.

\bibitem[Los10]{Los}
J\'er\^ome Los.
\newblock Infinite sequence of fixed-point free pseudo-anosov homeomorphisms.
\newblock {\em Ergodic Theory and Dynamical Systems}, 30(6):1739–--1755,
  2010.

\bibitem[Ni07]{Nifibered}
Yi~Ni.
\newblock Knot floer homology detects fibered knots.
\newblock {\em Inventiones mathematicae}, 170(3):577--608, 2007.

\bibitem[Ni20]{Ni3}
Yi~Ni.
\newblock Exceptional surgeries on hyperbolic fibered knots, 2020.

\bibitem[Ni22]{Ni2}
Yi~Ni.
\newblock Knot floer homology and fixed points, 2022.

\bibitem[OS04a]{OSknot}
Peter Ozsv\'ath and Zolt\'an Szab\'o.
\newblock Holomorphic disks and knot invariants.
\newblock {\em Advances in Mathematics}, 186(1):58--116, 2004.

\bibitem[OS04b]{OSinvt}
Peter Ozsváth and Zoltán Szabó.
\newblock Holomorphic disks and topological invariants for closed
  three-manifolds.
\newblock {\em Annals of Mathematics}, 159(3):1027--1158, 2004.

\bibitem[OS05]{OScont}
Peter Ozsv{\'a}th and Zolt{\'a}n Szab{\'o}.
\newblock {Heegaard Floer homology and contact structures}.
\newblock {\em Duke Mathematical Journal}, 129(1):39 -- 61, 2005.

\bibitem[Tan11]{Tange}
Motoo Tange.
\newblock On the non-existence of l-space surgery structure.
\newblock {\em Osaka J. Math}, 48:541--547, 2011.

\bibitem[Thu88]{Thurston}
William~P. Thurston.
\newblock On the geometry and dynamics of diffeomorphisms of surfaces.
\newblock {\em Bulletin (New Series) of the American Mathematical Society},
  19(2):417--431, 1988.

\bibitem[Thu98]{ThurstonFibered}
William~P. Thurston.
\newblock Hyperbolic structures on 3-manifolds, ii: Surface groups and
  3-manifolds which fiber over the circle, 1998.

\end{thebibliography}

\end{document}